\documentclass[11pt]{article}
\usepackage{amsmath}
\usepackage{amsfonts}
\usepackage{graphicx}
\usepackage{latexsym}
\usepackage{url}
\usepackage{color} 
\usepackage{dsfont}
\usepackage{amssymb}
\usepackage{amsthm}
\usepackage{comment}
\usepackage{enumerate}
\newif\ifcol
\coltrue 
\ifcol
\newcommand{\colorr}{\color[rgb]{0,0,0}}

\newcommand{\colory}{\color{yellow}}

\else
\newcommand{\colr}{\color[rgb]{0.8,0,0}}
\newcommand{\colorr}{\color{black}}

\newcommand{\colory}{\color{black}}

\fi

\excludecomment{en-text}
\includecomment{jp-text}
\excludecomment{comment}
\def\bd{\begin{description}}
\def\ed{\end{description}}
\def\im{\item}
\def\beas{\begin{eqnarray*}}
\def\eeas{\end{eqnarray*}}
\def\bi{\begin{itemize}}
\def\ei{\end{itemize}}
\def\sfa{{\sf a}}
\def\sfc{{\sf c}}
\def\sfd{{\sf d}}
\def\sfq{d}

\def\dote{\stackrel{\circ}{e}}
\def\tti{{\tt i}}

\def\cale{{\cal E}}
\def\calf{{\cal F}}

\def\calm{{\cal M}}
\def\caln{{\cal N}}

\def\cals{{\cal S}}

\def\calv{{\cal V}}

\def\bX{\overline{X}}
\def\bV{\overline{V}}
\def\bx{\overline{x}}
\def\bSigma{\overline{\Sigma}}
\def\yeq{\>=\>}

\newtheorem{remark}{Remark}

\newcommand{\complexi}{{\tt i}}
\newcommand{\iu}{{\tt i}u}
\newcommand{\iv}{{\tt i}v}
\newcommand{\koko}{{\colory koko}}

\newcommand{\R}{\mathbb{R}}

\newcommand{\E}{\mathbb{E}}
\newcommand{\PP}{\mathbb{P}}
\newcommand{\N}{\mathbb{N}}

\newcommand{\bbD}{{\mathbb D}}
\newcommand{\bbE}{{\mathbb E}}
\newcommand{\bbF}{{\mathbb F}}
\newcommand{\bbG}{{\mathbb G}}
\newcommand{\bbH}{{\mathbb H}}
\newcommand{\bbI}{{\mathbb I}}

\newcommand{\bbL}{{\mathbb L}}

\newcommand{\bbN}{{\mathbb N}}

\newcommand{\bbP}{{\mathbb P}}

\newcommand{\bbR}{{\mathbb R}}

\newcommand{\bbX}{{\mathbb X}}

\newcommand{\bbZ}{{\mathbb Z}}

\newcommand{\si}{\sigma}

\newcommand{\ep}{\varepsilon}

\def\nn{\nonumber}

\newcommand{\toop}{\stackrel{\PP}{\longrightarrow}}
\newcommand{\ucp}{\stackrel{u.c.p.}{\longrightarrow}}
\newcommand{\schw}{\stackrel{d}{\longrightarrow}}

\newcommand{\stab}{\stackrel{d_{st}}{\longrightarrow}}

\newcommand{\bee}{\begin{equation}}
\newcommand{\eee}{\end{equation}}
\newcommand{\bea}{\begin{eqnarray}}
\newcommand{\eea}{\end{eqnarray}}
\newcommand{\bean}{\begin{eqnarray*}}
\newcommand{\eean}{\end{eqnarray*}}
\def\half{\frac{1}{2}}

\renewcommand{\theequation}{\arabic{section}.\arabic{equation}}

\newtheorem{prop}{Proposition}[section]
\newtheorem{cor}[prop]{Corollary}

\newtheorem{lem}[prop]{Lemma}

\newtheorem{theo}[prop]{Theorem}

\def\edgeworth{Edgeworth }

\begin{document}

\title{Edgeworth expansion for Euler approximation of continuous diffusion processes}
\author{Mark Podolskij\thanks{Department of Mathematics, Aarhus University,  Ny Munkegade 118, 8000 Aarhus,
Denmark, Email: mpodolskij@math.au.dk.} 
\and Bezirgen Veliyev\thanks{CREATES, Department of Economics and Business Economics, Aarhus University,  Fuglesangs Alle 4, 8210 Aarhus V,
Denmark, Email: bveliyev@econ.au.dk}
\and Nakahiro Yoshida\thanks{Graduate School of Mathematical
Science, 3-8-1 Komaba, Meguro-ku, Tokyo 153, Japan, Email: nakahiro@ms.u-tokyo.ac.jp}}

\maketitle

\begin{abstract}
In this paper we present the Edgeworth expansion for the 
Euler approximation scheme of a continuous diffusion process driven by a Brownian motion.
Our methodology is based upon a recent work \cite{Yoshida2013}, which establishes 
Edgeworth expansions associated with asymptotic mixed normality using elements of Malliavin 
calculus. Potential applications of our theoretical results include higher order expansions 
for weak and strong approximation errors associated to the Euler scheme, and for studentized version
of the error process.

\ \

\bigskip


\end{abstract}




\section{Introduction} \label{sec1}
\setcounter{equation}{0}
\renewcommand{\theequation}{\thesection.\arabic{equation}}
In this work we consider a one-dimensional continuous stochastic process  
$(X_t)_{t \in [0,1]}$ that 
satisfies the stochastic differential equation 
\begin{align} \label{sde}
dX_t=a(X_t) dt+ b(X_t) dW_t \qquad \text{with} \qquad X_0=x_0,
\end{align} 
where $(W_t)_{t \in [0,1]}$ is a Brownian motion, defined on a filtered probability space $\left(\Omega, \mathcal F, (\mathcal F_t)_{t \in [0,1]}, \mathbb P \right)$. A simple and effective numerical scheme for 
the solution of \eqref{sde} is the Euler approximation scheme, which is given as follows. Let  
$\varphi_n: \R_+ \to \R_+$ be the function defined by $\varphi_n(t)=i/n$ when $t \in [i/n, (i+1)/n)$. 
The continuous Euler approximation scheme is described by  
\begin{align} \label{Eulersde}
dX_t^n=a\left(X_{\varphi_n (t)}^n \right) dt+ 
b\left(X_{\varphi_n (t)}^n \right)dW_t \qquad \text{with} \qquad X_0^n=x_0.
\end{align} 
The probabilistic properties of the Euler approximation scheme have been investigated in numerous 
papers. We refer to the classical work \cite{BT1,BT2, JP,KP, KuPr, MP} among many others. 
Asymptotic results in the framework of non-regular coefficients can be found in e.g. \cite{A, CS,HK,Yan}.

In this paper we are aiming to derive an Edgeworth expansion for the error process 
\begin{align} \label{error}
U^n = X^n - X.
\end{align}
Let us recall the classical convergence result for $(U^n_t)_{t \in [0,1]}$ from \cite{JP}.  
\begin{theo} \label{th1} \cite[Theorem 1.2]{JP}
Assume that the functions $a,b$ are globally Lipschitz and $a,b \in C^1(\R)$. Then we obtain the stable
convergence
\begin{align} \label{convergence}
V^n := \sqrt{n} U^n \stab V \qquad \text{on } C([0,1]) 
\end{align}
equipped with the uniform topology, where $V=(V_t)_{t\in [0,1]}$ is the unique solution of the stochastic differential equation 
\begin{align} \label{limitsde}
dV_t = a'(X_t) V_t dt + b'(X_t) V_t dW_t - \frac{1}{\sqrt{2}} bb'(X_t) dB_t \qquad \text{with} \qquad
V_0=0,
\end{align}
and $(B_t)_{t \in [0,1]}$ is a new Brownian motion defined on an extension of the probability space
$\left(\Omega, \mathcal F, (\mathcal F_t)_{t \in [0,1]}, \mathbb P \right)$ and independent of 
the $\sigma$-field $\mathcal F$.
\end{theo}
We will see later that the limiting process $V$ is an $\mathcal F$-conditional Gaussian martingale 
with $\mathcal F$-conditional zero mean. In particular, for each $t>0$, $V_t$ has a mixed normal distribution.  The aim of this work is to derive an Edgeworth expansion associated with Theorem \ref{th1}.  More specifically, for any regular $q$-dimensional random variable $F$ and any given
times $0< T_1 <\ldots <T_k \leq 1$, we would like to determine the function $p_n: \R^k
\times \R^q \to \R$ such 
that it holds 
\begin{align} \label{2Edgeworth}
\sup_{f \in \mathcal C_{q,k}} \left| \E[f(V^n_{T_1}, \ldots, V^n_{T_k},F)] 
- \int_{\R^k \times \R^q} f(z,x) p_n(z,x) dzdx \right| = o(1/\sqrt{n})
\end{align} 
for a large class of functions $ \mathcal C_{q,k}$. The methodology is based upon the work of 
Yoshida \cite{Yoshida2013}, which applies Malliavin calculus and stable convergence to obtain the Edgeworth
expansion associated with mixed normal limits. Another key ingredient in the derivation of 
\eqref{2Edgeworth} is the stochastic expansion of the error process $U^n$ and a non-degeneracy condition, which turns out to 
be rather complex in the case $k>1$. Related articles 
include \cite{DalalyanYoshida2011, PY, PVY}, which have studied Edgeworth expansions associated 
to covariance estimators, power variations and the pre-averaging estimator.

The paper is structured as follows. Section  \ref{sec2.1} presents various definitions and notation.  
Sections \ref{sec2.2} and \ref{sec2.3} are devoted to derivation of Edgeworth expansion for multivariate weighted quadratic functionals, which plays a crucial role in the asymptotic analysis of the Euler scheme. In Section \ref{sec3} we investigate the second order stochastic
expansion of the standardised error process associated with the Euler approximation scheme. 
The Edgeworth expansion for the error process is investigated in Section \ref{sec5}. Section \ref{sec6}
is devoted to several applications of our theoretical results, including asymptotic expansion of the weak and strong approximation errors, and density expansion for the studentized version of the error process. Some proofs are presented in Section
\ref{sec7}.

\section{Background} \label{sec2}
\setcounter{equation}{0}
\renewcommand{\theequation}{\thesection.\arabic{equation}}

\subsection{Definitions and notation} \label{sec2.1}
In this subsection we introduce basic notation,  some elements of Malliavin calculus and the definition of stable convergence in law.   

All vectors $x\in \R^k$ are understood as column vectors; $\|x\|$ stands for Euclidean norm of $x$ and
$x^{\star}$ denotes the transpose of $x$. 
For $x\in \R^k$ and $m \in \mathbb{Z}_+^k$ we set $x^m:=\prod_{j=1}^k x_j^{m_j}$
and $|m|=\sum_{j=1}^k m_j$. For any function
$f: \R\rightarrow \R$ we denote by $f^{(l)}$ its $l$th derivative; for a function $f:\R^k 
\times \R^q \rightarrow \R$ and 
$\alpha =(\alpha_1,\alpha _2)\in \mathbb{Z}_+^{k} \times 
\mathbb{Z}_+^{q}$ the operator $d^\alpha$ is defined via 
$d^\alpha= d^{\alpha_1}_{x_1} d^{\alpha_2}_{x_2}$. The set $C_p^l(\R^k)$ (resp. $C_b^l(\R^k)$)
denotes the space of  $l$ times continuously differentiable functions $f: \R^k \rightarrow \R$ such that all derivatives up to order
$l$ have polynomial growth (resp. are bounded).  
For a matrix $A \in \R^{k \times k}$ and a vector
$x\in \R^k$ we write $A[x^{\otimes 2}]$ to denote the quadratic form $x^{\star} Ax$; similarly,
for $x,y \in \R^k$ we write
$y[x]$ for the linear form $y^{\star}x$. 
Finally, $\complexi:=\sqrt{-1}$.

We now introduce some notions of Malliavin calculus (we refer to the books of Ikeda and Watanabe \cite{IkedaWatanabe1989} 
and Nualart \cite{N} for a detailed exposition of Malliavin calculus). 
The set $\mathbb L^p$ denotes the space of random variables with finite $p$th moment and 
we use the notation $\mathbb{L}_{\infty-}=\cap_{p>1} \mathbb{L}^p$; the corresponding 
$\mathbb L^p$-norms are denoted
by $\|\cdot\|_{\mathbb L^p}$.
Define $\mathbb H=\mathbb L^2([0,1], dx)$ and let 
$\langle \cdot, \cdot \rangle_{\mathbb H}$ denote the usual scalar product on $\mathbb H$. We denote by $D^l$ 
the $l$th Malliavin derivative operator and by $\delta^l$ its unbounded adjoint (also called Skrokhod integral of order $l$). 
The space $\mathbb D_{l,p}$ is the completion of the set of smooth random variables with respect to the norm
\[
\|Y \|_{l,p} := \left(\E[|Y|^p] + \sum_{m=1}^l \E[\|D^m Y\|_{\mathbb H^{\otimes m}}^p]\right)^{1/p}.
\]  
For any smooth $k$-dimensional random variable $Y$ the Malliavin matrix 
is defined via $\si_Y:=(\langle DY_i, DY_j \rangle_{\mathbb H})_{1\leq i,j\leq k}$. We  write $\Delta_Y:=\text{det }\si_Y$
for the determinant of the Malliavin matrix.
Finally, we set $\mathbb D_{l,\infty}= \cap_{p\geq 2} \mathbb D_{l,p}$. We sometimes use the notation
$\mathbb D_{l,p}(\R^k)$ to denote the space of all $k$-dimensional random variable $Y$ such 
that $Y_i \in \mathbb D_{l,p}$.

We use the notation $Y_n \stab Y$  to denote the stable convergence
in law. We recall that a sequence of random variables $(Y_n)_{n\in \N}$ 
defined on $(\Omega, \mathcal F, \mathbb P)$ with values in a metric space $E$ is said 
to converge stably with limit $Y$, written $Y_n \stab Y$, where $Y$ is defined on an extension
$(\overline \Omega, \overline{\mathcal F}, 
\overline{\mathbb P})$ of the original probability space $(\Omega, \mathcal F, \mathbb P)$, 
iff for any bounded, continuous function $g$ and any bounded $\mathcal{F}$-measurable random variable $Z$ it holds that
\begin{equation} \label{defstable}
\E[ g(Y_n) Z] \rightarrow \overline{\E}[ g(Y) Z], \quad n \rightarrow \infty.
\end{equation}  
The notion of stable convergence is due to  Renyi \cite{REN}. We also refer to \cite{AE} for properties
of this mode of convergence.

Finally, for two vector fields $V_0$ and $V_1$ we denote by
$\text{Lie}[V_0;V_1]$  the Lie algebra generated by $V_1$ and $V_0$. 
That is, $\text{Lie}[V_0;V_1]=\text{span}\left(\bigcup_{j=0}^\infty\Sigma_j\right)$, where 
 $\Sigma_0=\{V_1\}$ and $\Sigma_j=\{[V,V_i];\> V\in\Sigma_{j-1},\>i=0,1\}$ ($j\geq1$) 
 with the Lie bracket $[\cdot,\cdot]$. $\text{Lie}[V_0;V_1](x)$ stands for 
 $\text{Lie}[V_0;V_1]$ evaluated at $x$.

\subsection{Edgeworth expansion associated with mixed normal limits: The quadratic
case} \label{sec2.2}

In this subsection we will study the (second order) Edgeworth expansion associated with certain quadratic functionals of Brownian motion that will be crucial for the treatment of the error process $V^n$. Indeed, we will see later that the dominating martingale term in the expansion of $V^n$ has a quadratic form. 
 The results are similar in spirit to   \cite[Theorem 4]{Yoshida2013}, but we will require quite different non-degeneracy arguments.

On a filtered Wiener space $(\Omega, \mathcal F, (\mathcal F_t)_{t\in [0,1]},
\mathbb P)$ we consider a $k$-dimensional random functional $Z_n$, which admits the decomposition
\begin{equation} \label{vndec}
Z_n=M_n + n^{-1/2} N_n, 
\end{equation}
where $M_n$ and $N_n$ are tight sequences of random variables.  
We assume that $M_n$, which will have a quadratic form, converges stably in law to a mixed normal variable $M$: 
\begin{equation} \label{mnstab}
M_n \stab M,
\end{equation}
where the random variable $M$ is defined on an extension $(\overline \Omega, \overline{\mathcal F}, 
\overline{\mathbb P})$ of the original probability space $(\Omega, \mathcal F, \mathbb P)$ and, conditionally on $\mathcal F$,
$M$ has a normal law with mean $0$ and conditional covariance matrix $C\in \R^{k \times k}$. In this case we use the notation
\begin{equation*}
M\sim MN(0, C).
\end{equation*}
For concrete applications it is often useful to consider the Edgeworth expansion for the pair
$(Z_n, F_n)$, where $F_n$ is another $q$-dimensional random functional satisfying the convergence
in probability
\begin{equation*}
F_n \toop F.
\end{equation*}
Obviously, such a framework is important when the statistic at hand does not only depend on the sequence $Z_n$, but also on an external random variable $F$ (in this case we may set $F_n=F$). 
In the statistical context the most useful application is the case where $F_n \toop C$. In this situation 
we obtain by properties of stable convergence that   
\begin{equation*}
F_n^{-1/2} Z_n \schw \mathcal N_k(0,\text{id}_k)
\end{equation*}
when $F_n \in \R^{k\times k}$ is positive definite and $\text{id}_k$ denotes the identity matrix. Thus, the asymptotic expansion of the law of $(Z_n, F_n)$ would imply the Edgeworth expansion for the studentized 
statistic $F_n^{-1/2} Z_n$.

In the next step we embed the previous static framework into a martingale setting. 
We assume that the leading term $M_n$ is a terminal 
value of some continuous $(\mathcal F_t)$-martingale $(M_t^n)_{t\in [0,1]}$, that is $M_n=M_1^n$.
We also consider  stochastic processes $(M_t)_{t\in [0,1]}$ and $(C^n_t)_{t\in [0,1]}$ with
values in $\R^{k}$ and $\R^{k \times k}$ respectively, such that
\begin{equation}
M=M_1, \qquad C_t = \langle M \rangle_t, \qquad C_t^n = \langle M^n \rangle_t, \qquad C_n= \langle M^n \rangle_1.
\end{equation}
Here the process $(M_t)_{t\in [0,1]}$, defined on extended 
probability space $(\overline \Omega, \overline{\mathcal F}, 
\overline{\mathbb P})$, represents the stable limit of the continuous $(\mathcal F_t)$-martingale $(M_t^n)_{t\in [0,1]}$,
while $C^n$ denotes the quadratic covariation process associated with $M^n$. 

Now, we shall introduce a particular type of quadratic functionals. For a sequence of time points 
$(T_j)_{1\leq j\leq k}$ not depending on $n$  
with $0<T_1<\ldots<T_k$,  
we consider a sequence of partitions $\pi^n=(t_i)_{1\leq i \leq m_n}$ of $[0,1]$ 
such that $0=t_0<t_1<\ldots<t_{m_n}$ and that $\{T_j\}_{1\leq j\leq k} \subset\{t_i\}_{1\leq i \leq m_n}$ for every $n\in\bbN$. Here $t_j$ may depend on $n$ though we omit $n$ for notational simplicity. 
Let $I_i=[t_{i-1},t_i)$ and $|I_i|=t_i-t_{i-1}$. Suppose that $n^4\sum_{i=1}^{m_n}|I_i|^5=O(1)$ as $n\to\infty$. Next, we consider a strongly predictable  kernel 
$K^n=(K^{n,j})_{1\leq j\leq k}:\Omega\times[0,1]\to\bbR^k$ satisfying 
\[
K^{n,j}(t)=K^{n,j}(t_{i-1}) \text{ for } t\in I_i \qquad \text{and} \qquad  K^{n,j}(t)=0 \text{ if } 
t\geq T_j.
\]
The aforementioned sequence of quadratic type martingales $M^n=(M^{n,j})_{1\leq j \leq k}$ is defined 
by 
\begin{align} \label{Mnj}
M^{n,j}_t
= \sqrt{n}\sum_{i=1}^{m_n}
K^{n,j}(t_{i-1})\int_{t_{i-1}\wedge t}^{t_i\wedge t} \int_{t_{i-1}}^s
dW_r dW_s, \qquad t \in [0,1].
\end{align}
Let $K:\Omega\times[0,1]\to\bbR$ be a continuous adapted process and set
\bea\label{180524-1}
\bbI_s^j= \half\int_s^{T_j} K(r)^2 dr,\qquad s\in(T_{j-1},T_j]. 
\eea
Our first set of conditions relates the kernel $K^n$ to $K$ and introduces some integrability assumptions, which are similar in spirit to 
assumptions imposed in \cite{Yoshida2013}.
Recall that $F \in \R^q$, set $\ell=k+q+8$ and 
{\colorr let $\frac{1}{3}<d<\frac{1}{2}$. }
\bd
\im[(B1)] {\bf (i)} $K^n(t)\in\bbD_{\ell+1,\infty}(\bbR^k)$ and 
there exists a density $D_{r_1,...,r_m}K^n(t)$ representing each derivative such that 
\beas 
\sup_{r_1,...,r_m\in(0,1),\atop t\in[0,1],\>n\in\bbN} 
\big\| D_{r_1,...,r_m}K^n(t)\big\|_{\mathbb L^p} &<& \infty
\eeas
for every $p>1$ and $m=0,1,\ldots,\ell+1$. 
\bd
\im[(ii)] For every $p>1$ and $j=1,\ldots,k$, 
\beas 
\sup_{1 \leq i \leq m_n } \sup_{t\in (t_{i-1},t_i)}\big\|K^{n,j}(t)-K(t)1_{\{t< T_j\}}\big\|_{\ell,p} &=& 
O(n^{-{\colorr d}})
\eeas
as $n\to\infty$. 

\im [(iii)] For every $p>1$ and $j=1,...,k$, 
\beas
\sup_{s\in(T_{j-1},T_j)}
\left\|\left[\frac{\bbI^j_s}{T_j-s}\right]^{-1}\right\|_{\mathbb L^p}  &<& \infty.
\eeas
\ed
\ed
From (B1)(i), (ii) we deduce that 
\[
C^{n,j}_t=\langle M^{n,j}\rangle_t \toop C^{j}_t  =\half\int_0^{t\wedge T_j} K(s)^2 ds.
\]
Furthermore, (B1)(iii) implies 
\bea\label{290219-1}
\big(C^j_{T_j}-C^j_{T_{j-1}}\big)^{-1}
&\in& \mathbb{L}_{\infty-}
\eea
for $j=1,\ldots,k$. 
In particular, $\det C^{-1}\in \mathbb{L}_{\infty-}$ for $C=(C^{j_1\wedge j_2}_1)_{1\leq j_1,j_2\leq k}$.

Now, let us set 
\begin{align} \label{cnfn}
\widehat{C}_n= \sqrt{n} (C_n -C), \qquad \widehat{F}_n= \sqrt{n}(F_n -F),
\end{align}
where $C_n=C^n_1$ with $C^n_t = (\langle M^{n,j_1},M^{n,j_2} \rangle)_{1\leq j_1,j_2\leq k}$.
In the validation of the asymptotic expansion a truncation functional $s_n:\Omega\to \bbR^k$ will play 
an important  role; see Section \ref{180522-9} for its explicit definition. We set  $\ell_*=2[q/2]+4$ and present the next set of assumptions that determines 
the asymptotic distribution of the vector $(M^n_{\cdot}, N_n, \widehat{C}_n, \widehat{F}_n)$ along with some new integrability conditions.  

\bd
\im[(B2)] {\bf (i)} $F\in\bbD_{\ell+1,\infty}(\bbR^{q})$,
$\sup_{r_1,...,r_m\in(0,1)}\|D_{r_1,...,r_m}F\|_{\mathbb L^p} <\infty$ for every $p>1$ and 
$m=1,...,\ell+1$. Moreover $r\mapsto D_rF$ and 
$(r,s)\mapsto D_{r,s}F$ ($r\leq s$) are continuous a.s. 
\bd
\im[(ii)] $F_n\in\bbD_{\ell+1,\infty}(\bbR^{q})$, $N_n\in\bbD_{\ell+1,\infty}(\bbR^k)$ 
and $s_n=(s^j_n)\in\bbD_{\ell,\infty}(\bbR^k)$. Moreover, 
\beas
\sup_{n\in\bbN}\bigg\{
\big\|\widehat{C}_n\big\|_{\ell,p}
+\big\|\widehat{F}_n\big\|_{\ell+1,p}
+\big\|N_n\big\|_{\ell+1,p}
+\big\|s_n\big\|_{\ell,p}\bigg\}
&<& \infty
\eeas
for every $p>1$. 
\im[(iii)] 
$(M^n,N_n,\widehat{C}_n,\widehat{F}_n)\stab
(M,N,\widehat{C},\widehat{F})$ 
for a random vector $(M,N,\widehat{C},\widehat{F})$ defined on an extension of $(\Omega,\calf,P)$. 
\im[(iv)] 
For $u\in\bbR^k$ and $v\in\bbR^{q}$, 
the conditional expectations 
$\mathbb{E}\big[\widehat{C}|M_1=z\big][u^{\otimes2}]$, $\mathbb{E}\big[\widehat{F}|M_1=z\big][v]$ 
and $\mathbb{E}\big[N|M_1=z\big]$ are in 
the polynomial ring $\bbD_{\ell_*,\infty}(\bbR)[z]$ (with coefficients in $\bbD_{\ell_*,\infty}(\bbR)$). 

\ed
\ed
Finally, we will require a non-degeneracy condition  on the pair  $(M^n_t,F)$. 
Let us introduce the process
\beas 
\bbX^j_t &=& 
(M^{n,1}_1,...,M^{n,j-1}_1,M^{n,j}_t,F). 
\eeas

\bd
\im[(B3)] {\bf (i)} For each $j=1,...,k$, there exists a sequence 
$(\tau^j_n)_{n\in\bbN}\subset(T_{j-1},T_j)$ 
such that $\sup_n\tau^j_n<T_j$ and that 
\beas 
\sup_{t\in[\tau^j_n,T_j]} \mathbb P\big[\det\sigma_{\bbX^j_t}<s^j_n\big] &=& O(n^{-\nu})
\eeas
for some $\nu>\ell/6$. 
\bd
\im[(ii)] 
$\limsup_{n\to\infty} \mathbb E\big[(s^j_n)^{-p}\big]<\infty$ for every $p>1$ and $j=1, \ldots, k.$
\ed
\ed

\subsection{Random symbols $\underline \sigma,\overline \sigma$ and the main result} \label{sec2.3}
In order to present the \edgeworth expansion for the pair $(Z_n, F_n)$ we need to define two {\it random symbols} 
$\underline \sigma$ and $\overline \sigma$, which play a crucial role in what follows. We call $\underline \sigma$ the adaptive
(or classical) random symbol and $\overline \sigma$ the anticipative random symbol. The adaptive random symbol $\underline{\sigma}$ is defined by 
\beas 
\underline{\sigma}(z;\iu,\iv) 
&=&
\half \mathbb E\big[\widehat{C}|M_1=z\big][(\iu)^{\otimes2}]
+\mathbb E\big[N|M_1=z\big][\iu]+ \mathbb E\big[\widehat{F}|M_1=z\big][\iv].
\eeas
Let $\overline{K}(t)=(K(1)1_{\{t<T_j\}})_{1\leq j \leq k}$. 
The anticipative random symbol $\overline{\sigma}$ is defined by 
\bea\label{170207-5}
\overline{\sigma}(iu,iv) 
&=&
\half \int_0^1 \overline{K}(t)[\iu]\sigma_t(\iu,\iv)dt
\eea
where 
\beas 
\sigma_t(\iu,\iv) 
&=&
\bigg(-\half D_tC[u^{\otimes2}]+D_tF[\iv]\bigg)^2
\\&&
+\bigg(-\half D_tD_tC[u^{\otimes2}]+D_tD_tF[\iv]\bigg).
\eeas
The derivative $D_tD_t$ stands for $\lim_{s\uparrow t}D_sD_t$. 
The full random symbol is defined by 
\beas\label{180520-2}
\sigma &= &\underline{\sigma}+\overline{\sigma}
\eeas
and has the form
\bea\label{180520-3}
\sigma(z;\iu,\iv) &=& \sum_\alpha c_\alpha(z) (\iu)^{\alpha_1}(\iv)^{\alpha_2}
\eea
where 
$\alpha=(\alpha_1,\alpha_2)\in\bbZ_+^k\times\bbZ_+^{q}$.

Under conditions of the previous subsection, the non-degeneracy of $F$ is ensured and it has a differentiable density function $p^F$. 
Thus, the following function $p_n$ is well defined: 
\begin{align} \label{pn}
p_n(z,x) 
&=
\mathbb E\big[\phi(z;0,C)|F=x]p^F(x)
\\
&+n^{-1/2} \sum_\alpha (-d_z)^{\alpha_1}(-d_x)^{\alpha_2}
\bigg\{ \mathbb E\big[c_\alpha(z)\phi(z;0,C)|F=x\big]p^F(x)\bigg\}. \nonumber 
\end{align}
For positive numbers $R$ and $\gamma$, 
$\cale(R,\gamma)$ denotes the set of measurable functions $f:\bbR^{k+q}\to\bbR$ 
such that $|f(z,x)|\leq R(1+|z|+|x|)^\gamma$ for all $z\in\bbR^k$ and $x\in\bbR^{q}$. 
The error of the approximation of the distribution of $(Z_n,F_n)$ by $p_n$ is evaluated by 
the quantity
\beas 
\Delta_n(f) &=& 
\bigg| \mathbb{E}\big[f(Z_n,F_n)\big]-\int f(z,x)p_n(z,x)dzdx\bigg|
\eeas
for $f\in\cale(R,\gamma)$. 
The main result of this section is the following. 
\begin{theo}\label{20170206-1} 
Suppose that $Z_n$ is given by (\ref{vndec}) with 
$M_n$ defined by (\ref{Mnj}).
Suppose that $(B1)$, $(B2)$ and $(B3)$ are satisfied. Then 
\bea\label{180520-5}
\sup_{f\in\cale(R,\gamma)}\Delta_n(f)=o(n^{-1/2})
\eea
as $n\to\infty$ for any positive numbers $R$ and $\gamma$. 
\end{theo}

\section{Stochastic expansion of the error process} \label{sec3}
\setcounter{equation}{0}
\renewcommand{\theequation}{\thesection.\arabic{equation}}
In this section we derive  explicit expressions for the first and second order approximation of the normalised 
error process $V^n$. The following well known lemma, which presents an explicit solution 
of an affine stochastic differential equation, will be a helpful tool.

\begin{lem} \label{lem1}
Assume that $(Y_t)_{t\in [0,1]}$  is the unique strong solution of the stochastic differential
equation
\begin{align} \label{linearsde}
dY_t= (c_t Y_t +  \tilde{c}_t)dt+ (d_t Y_t +  \tilde{d}_t) dW_t \qquad \text{with} \qquad Y_0=y_0,
\end{align} 
where $(c_t)_{t\in [0,1]}, (\tilde c_t)_{t\in [0,1]}, (d_t)_{t\in [0,1]}, (\tilde d_t)_{t\in [0,1]}$ are predictable stochastic processes. 
Then the process $(Y_t)_{t\in [0,1]}$ exhibits an explicit solution given by 
\begin{align} \label{explicitsde}
Y_t &= \Sigma_t \left[ y_0 +   \int_0^t   \Sigma_s^{-1}  \left((\tilde{c}_s - d_s \tilde{d}_s)ds +  \tilde{d}_s dW_s \right) \right], \\
  \Sigma_t &= \exp \left( \int_0^t d_s dW_s +   \int_0^t \left(c_s - \frac12 d_s^2\right) ds \right). \nonumber
\end{align}  
\end{lem} 

\begin{proof}
The proof follows a classical route for solutions of inhomogeneous differential  equations. First, we recall that 
the process $\Sigma$ satisfies the stochastic differential equation 
$d\Sigma_t= c_t \Sigma_t dt+ d_t \Sigma_t dW_t$, which is shown 
by applying It\^o's formula to the function $f(x,y)=\exp(x+y)$.  Now, setting 
$Z_t= y_0+\int_0^t   \Sigma_s^{-1}  (\tilde{c}_s - d_s \tilde{d}_s)ds +  \int_0^t   \Sigma_s^{-1}  \tilde{d}_s dW_s $, we conclude by the product formula
that 
\begin{align*}
Y_t&=y_0+ \int_0^t \Sigma_s dZ_s +  \int_0^t Z_s d\Sigma_s + \langle Z, \Sigma \rangle_t \\[1.5 ex]
&= y_0+  \int_0^t   (\tilde{c}_s - d_s \tilde{d}_s) ds +  \int_0^t    \tilde{d}_s dW_s + 
 \int_0^t c_s Z_s \Sigma_s ds +  \int_0^t d_s Z_s \Sigma_s dW_s + 
 \int_0^t   d_s \tilde{d}_s ds \\[1.5 ex]
 &= y_0 + \int_0^t  (c_s Y_s + \tilde{c}_s) ds + \int_0^t  (d_s Y_s + \tilde{d}_s) ds
\end{align*}   
and the proof is complete.
\end{proof}
Applying the same type of proof as in Lemma \ref{lem1}, we deduce that the limiting process $V$
introduced at \eqref{limitsde} can be written explicitly 
as 
\begin{align*}
V_t = -\frac{1}{\sqrt 2}\Sigma_t  \int_0^t   \Sigma_s^{-1}  bb'(X_s) dB_s,
\end{align*}
where the process $(\Sigma_t)_{t\geq 0}$ is defined by 
\begin{align} \label{vexplicit}
\Sigma_t= 
\exp \left( \int_0^t b'(X_s) dW_s +   \int_0^t \left(a' - \frac12 (b')^2 \right)(X_s) ds \right).
\end{align}
Since the process $\Sigma$ is $\mathcal F$-measurable, we see that $V$ is an $\mathcal F$-conditional
Gaussian martingale with $\mathcal F$-conditional mean zero.

In the first step we will obtain an explicit representation of the leading term of the normalised error process $V^n$ defined at \eqref{convergence}. This stochastic expansion can be 
also found in the proof of \cite[Theorem 1.2]{JP}. Nevertheless, we will prove this result for the sake of completeness.  

\begin{theo} \label{th2n}
Let us consider the process 
\begin{align} \label{barvn}
\overline{V}_t^n = -\sqrt{n} \Sigma_t  \int_0^t   \Sigma_{\varphi_n(s)}^{-1}  bb'(X^n_{\varphi_n(s)}) (W_s - W_{\varphi_n(s)}) dW_s,
\end{align}
where $\Sigma$ is defined in \eqref{vexplicit}. Then it holds that 
\[
\sup_{t \in [0,1]} |V_t^n -  \overline{V}_t^n| \toop 0.
\]
\end{theo}
We remark at this stage that the process $(\Sigma_t^{-1} \overline{V}_t^n)_{t \in [0,1]}$ is a continuous
martingale of quadratic form with random weights. Thus, second order Edgeworth expansion for the functional $\overline{V}_t^n$ can be deduced from the corresponding expansion for 
the pair $( \Sigma_t ,\Sigma_t^{-1} \overline{V}_t^n)$.  

In the next step we need to determine the second order stochastic expansion for the standardised 
error process $(V_t^n)_{t \in [0,1]}$. Apart from rather complex approximation techniques, the result
of Lemma \ref{lem1} is crucial for the next theorem.  We remark that this statement has an interest
in its own right.

\begin{theo} \label{th3}
Assume that the functions $a,b$ are globally Lipschitz and $a,b \in C^2(\R)$. Define the process
$(R_t^n)_{t \geq 0}$ via  
\begin{align} \label{rn}
dR_t^n &= \left(\frac{1}{2 \sqrt{n}} a''(X_t) (V_t^n)^2 + \sqrt{n} 
 b\left((b')^2 -a'\right)(X_{\varphi_n(t)}^n)
(W_t - W_{\varphi_n(t)})  \right.  \\[1.5 ex]
 &\left. -\sqrt{n}aa'(X_{\varphi_n(t)}^n) 
(t - \varphi_n(t)) - \frac{\sqrt{n}}{2} b^2a''(X_{\varphi_n(t)}^n) (W_t - W_{\varphi_n(t)})^2
\right) dt\nonumber \\[1.5 ex]
&+ \Big(
\frac{1}{2 \sqrt{n}} b''(X_t) (V_t^n)^2 +\sqrt{n}\Big( b(b')^2
 -\frac{b^2b''}{2} \Big) (X_{\varphi_n(t)}^n)
(W_t - W_{\varphi_n(t)})^2 
  \nonumber  \\[1.5 ex]
 &    - \sqrt{n}ab'(X_{\varphi_n(t)}^n) (t - \varphi_n(t))
 \Big) dW_t = R_t^n(1) dt +  R_t^n(2) dW_t \nonumber
\end{align}
Then the process $\sqrt{n} R^n$ is tight and we have that 
\[
\sqrt{n} \sup_{t \in [0,1]} \left|V_t^n -  
\left(\overline{V}_t^n + \Sigma_t \int_0^t \Sigma_s^{-1}  \left(dR_s^n  - b'(X_s)R_s^n(2) ds\right)\right) \right| \toop 0 .
\]
\end{theo} 
Theorem \ref{th3} implies that, for any fixed $t\in[0,1]$, we have the stochastic expansion 
$V_t^n = \Sigma_t(M^n_t + n^{-1/2} N^n_t) $ with 
\begin{align} \label{mtn}
M^n_t=\Sigma_t^{-1} \overline{V}_t^n, \qquad N^n_t= \sqrt{n}
\int_0^t \Sigma_s^{-1}  \left(dR_s^n - b'(X_s)R_s^n(2) ds \right) + o_{\mathbb P}(1).
\end{align}
In the next section we will determine the stable central limit theorem 
for the triplet $(M^n_t, \sqrt{n}(C_t^n -C_t), N^n_t)_{t \in [0,1]}$.

\section{Stable central limit theorems and Edgeworth expansion} \label{sec5}
\setcounter{equation}{0}
\renewcommand{\theequation}{\thesection.\arabic{equation}}

\subsection{Central limit theorems} \label{sec5.1}
Having derived the stochastic expansion for the standardised error process $(V_t^n)_{t \in [0,1]}$
in the previous section, we now need to prove the stable central limit theorem required in assumpion
(B2)(iii). For this purpose we introduce the following auxiliary processes:

\begin{align}
A_t^n(1) &=  n \int_0^t \Sigma_{s}^{-1} b\left((b')^2 - a'\right)(X_{\varphi_n(s)}^n) 
\left(\varphi_n(s+n^{-1}) - s \right) dW_s \\[1.5 ex]
&-n \int_0^t  \Sigma_{s}^{-1}  ab'(X_{\varphi_n(s)}^n) 
(s - \varphi_n(s)) dW_s  \nonumber \\[1.5 ex]
& + n \int_0^t \Sigma_{s}^{-1} \left(b(b')^2
-\frac{b^2b''}{2} \right)(X_{\varphi_n(s)}^n)
(W_s - W_{\varphi_n(s)})^2  dW_s,  \nonumber \\[1.5 ex]
A_t^n(2) &= 2n^{3/2} \int_0^t   \left( \Sigma_{s}^{-1}  bb'(X_{\varphi_n(s)}^n) 
\right)^2 \left(\varphi_n(s+n^{-1}) - s \right) \left(W_s - W_{\varphi_n(s)} \right) dW_s, \\[1.5 ex]
A_t^n(3) &=  n  \int_0^t \Sigma_s^{-1} a\left((b')^2 - a' \right) (X_{\varphi_n(s)}^n) \left(s - \varphi_n(s) \right) ds \\[1.5 ex]
& - \frac{n}{2} \int_0^t \Sigma_s^{-1}
\left(b^2a'' + b^2b'b'' - 2b(b')^3
\right) (X_{\varphi_n(s)}^n) \left(W_s - W_{\varphi_n(s)} \right)^2 ds.  \nonumber
\end{align} 
Our first asymptotic result is the following stable central limit theorem.

\begin{prop} \label{firstclt} Assume that conditions of Theorem \ref{th3} are satisfied.
Then it holds that 
\begin{align} \label{firststab} 
L^n:=\left(M^n, A^n(1), A^n(2) \right) \stab L = \int_{0}^{\cdot} v_s dW_s + \int_{0}^{\cdot} (u_s -v_s^{\star}v_s)^{1/2} dB_s \quad \text{on } C([0,1])^3,
\end{align} 
where $(B_t)_{t \in [0,1]}$ is a $3$-dimensional Brownian motion defined on an extension 
 $(\overline \Omega, \overline{\mathcal F}, \overline{\mathbb P})$ of the original probability space and independent of $\mathcal{F}$,
 and the processes $v_s=(v_s^1, v_s^2,v_s^3)$, $u_s=(u_s^{ij})_{1\leq i,j\leq3}$ are defined by
 \begin{align*}
&v_s^1=v_s^3=0, \qquad v_s^2= \Sigma_{s}^{-1} \left( b(b')^2 -  \frac{ab' + a'b}{2} -  \frac{b^2b''}{4} \right) (X_s),
\\[1.5 ex]
& u_s^{12}=u_s^{21}=u_s^{23}=u_s^{32}=0, \\[1.5 ex]
& u_s^{11}= \frac 12  \left( \Sigma_{s}^{-1}  bb'(X_{s})\right)^2, \qquad u_s^{33}= \frac 13
\left( \Sigma_{s}^{-1}  bb'(X_{s})\right)^4, \qquad u_s^{13}= u_s^{31}= - \frac 13 \left( \Sigma_{s}^{-1}  bb'(X_{s})\right)^3,  \\[1.5 ex]
&u_s^{22} = \frac 13 \Sigma_{s}^{-2} \Big( 
b^2 \left[\left((b')^2 - a'\right)^2 + \left((b')^2
-\frac{bb''}{2} \right)\left(   4(b')^2 - \frac{3bb''}{2} -a' \right) \right]  \\[1.5 ex]
& +(ab')^2- abb' \left(3(b')^2-a' -bb'' \right) \Big) (X_s).
 \end{align*}
\end{prop}

\begin{proof}
Note that $L^n$ is a continuous martingale with mean zero. According to \cite[Theorem IX.7.3]{JS}, 
it is sufficient
to prove that
\begin{align*}
\langle L^n \rangle_t \toop \int_0^t u_s ds, \qquad \langle L^n, W \rangle_t \toop \int_0^t v_s ds,
\qquad \langle L^n, Q \rangle_t \toop 0, \qquad  \forall t\in [0,1],
\end{align*}
where the last statement should hold for any bounded continuous martingale $Q$ with $\langle W, Q \rangle=0$. The first two statements follow by a straightforward but tedious  computation taking into account that 
$X_s^n \toop X_s$ for any $s\in [0,1]$, $\sup_{s\in [0,1]} |\varphi_n(s)-s| \to 0$ and the continuity 
of involved processes/functions.  The third condition is a consequence of the formula  $\langle \int_0^{\cdot} w_s dW_s, Q \rangle_t =
\int_0^{t} w_s d\langle W, Q \rangle_s=0$ for any predictable process $(w_s)_{s\in [0,1]}$.
\end{proof}
As a consequence of the previous result we deduce the joint stable central limit theorem for the 
triplet $(M^n_t, N^n_t, \sqrt{n}(C_t^n -C_t))_{t \in [0,1]}$. 

\begin{prop} \label{secondclt}
Assume that conditions of Theorem \ref{th3} are satisfied.
Then it holds that 
\begin{align} \label{secondstab} 
(M^n, N^n, \sqrt{n}(C^n -C)) &\stab \left(L^1, \frac{1}{2} \int_0^{\cdot}
\Sigma_{s} \left(a'' +b''-b'b'' \right)(X_s) (L_s^1)^2 ds + L^2 + A(3) ,L^3 \right)  \nonumber \\[1.5 ex]
& := (M,N, \widehat{C}) \qquad \text{on } C([0,1])^3, 
\end{align}
where the process $L=(L^1,L^2,L^3)$ has been introduced in Proposition \ref{firstclt} and the
process $(A_t(3))_{t \in [0,1]}$ is defined as
\begin{align*}
A_t(3) &=  \int_0^t \Sigma_s^{-1} \left( 
\frac 12 a(b')^2 +\frac 12 b(b')^3 - \frac 12 aa' -
\frac{1}{4} a''b^2 - \frac{1}{4} b^2b' b''\right) (X_s)
ds.
\end{align*}
\end{prop}

\begin{proof}
First of all, it holds that $\sup_{t\in [0,1]} |A_t^n(3) -A_t(3)| \toop 0$, which is due to 
\cite[Theorem 7.2.2]{JP}.  Secondly, using the identities 
$(W_b-W_a)^2-(b-a)= 2 \int_a^b (W_s-W_a) dW_s$ and 
$\int_a^b (Y_s-Y_a) ds = \int_a^b (b-s) dY_s$, which hold for any $b>a$ and any continuous semimartingale $Y$, we obtain that 
\[
\sqrt{n}(C^n_t -C_t) = A_t^n(2). 
\]
Furthermore, observing the definition  \eqref{rn} of the process $R^n$, we deduce the identity
\begin{align*}
N^n_t &= \sqrt{n} \int_0^t \Sigma_s^{-1}  \left(dR_s^n - b'(X_s)R_s^n(2) ds \right) + o_{\mathbb P}(1 ) \\[1.5 ex]
&= A_t^n(1) + A_t^n(3) + \frac 12 \int_0^t \Sigma_{s}^{-1} (V_s^n)^2 \left( a''
+ b''-b'b'' \right)(X_s) ds + o_{\mathbb P}(1 ).
\end{align*}
Now, due to convergence \eqref{firststab} in Proposition \ref{firstclt} and the properties of stable 
convergence we deduce that $(M^n, A^n(1), A^n(2),A^n(3), \Sigma, X) \stab (L^1,L^2,L^3, A(3), \Sigma,X)$
on $C([0,1])^6$. Hence, by \cite[Theorem VI.6.22]{JS} and  continuous mapping theorem for stable convergence applied to the function $H: C([0,1])^6 \rightarrow C([0,1])^3$
\[
H(y):= \left(y_1, y_2+y_4 
+  \frac 12 \int_0^{\cdot} y_{5}(s)^{-1} (y_1(s)y_5(s))^2 \left( a''
+ b''-b'b'' \right)(y_6(s)) ds, y_3 \right) 
\]
we obtain  that
\[
(M^n, N^n, \sqrt{n}(C^n -C)) \stab (M,N, \widehat{C}) \qquad \text{on } C([0,1])^3. 
\]  
This completes the proof of Proposition \ref{secondclt}. 
\end{proof}
We remark that the $3$-dimensional limiting process $(M,N, \widehat{C})$ is an $\mathcal{F}$-conditional Gaussian martingale. This property will help us to compute the classical random symbol
$\underline \sigma (z, \iu, \iv)$ 
in the next section.

\subsection{Multivariate Edgeworth expansion associated with the Euler scheme}\label{MultiEdgeworth}

Let us now consider fixed time points $0=T_0<T_1<\ldots<T_k\leq 1$.
In this section we will investigate 
the multivariate Edgeworth expansion for the vector $(V^n_{T_1}, \ldots, V^n_{T_k})$. We recall
the representation introduced at \eqref{mtn}:
\begin{align*}
V_{T_j}^n &= \Sigma_{T_j} (M^n_{T_j} + n^{-1/2} N^n_{T_j} ) 
\qquad \text{with} 
\\[1.5 ex]
M^n_{T_j} &=\Sigma_{T_j} ^{-1} \overline{V}_{T_j} ^n, 
\qquad N^n_{T_j} = \sqrt{n}
\int_0^{T_j}  \Sigma_s^{-1}  \left(dR_s^n - b'(X_s)R_s^n(2) ds \right) 
+ o_{\mathbb P}(1).
\end{align*}

\noindent
According to the Edgeworth expansion theory demonstrated in Section \ref{sec2}, we will first derive the density expansion for the vector 
$(\Sigma_{T_j}, M^n_{T_j} + n^{-1/2} N^n_{T_j})_{1\leq j \leq k}$.

We define the $k$-dimensional 
$(\mathcal F_t)$-martingale with components 
$M^{n,j}:= (M^n_{\min (t, T_j)})_{t\in [0,1]}$, 
which obviously satisfies the terminal condition 
$M^{j,n}_1= M^n_{T_j}$ for $j=1,\ldots,k$.
Similarly, we set $N^{n,j}= N^n_{T_j}$. We introduce the set of 
increasing numbers
$(t_i)_{ 0\leq i \leq m_n}$
via $\{t_i\}=\{j/n:~j=0, \ldots, n\} \cup \{T_1, \ldots, T_k\}$.
In the 
notation of Section \ref{sec2.2} the martingale $M^{n,j}$ satisfies the representation 
\eqref{Mnj} with 
\begin{align} \label{settingw}
K^{n,j}(s) = -\Sigma_{\varphi_n(s)}^{-1} bb'(X^n_{\varphi_n(s)}) 1_{[0, T_j)}(\varphi_n(s))
\quad \text{and} \quad K(s) = -\Sigma_{s}^{-1} bb'(X_{s}). 
\end{align}
The anticipative random symbol $\overline{\sigma}$ is then defined through the identity \eqref{170207-5}. Now, we turn our attention to the adaptive random symbol $\underline{\sigma}$. 

We consider  a $ \overline{q}:=(k+ q)$-dimensional random variable 
\beas\label{180520-1}
G=(\Sigma_{T_1},
\ldots, \Sigma_{T_k}, F),
\eeas
where $F$ is a $q$-dimensional random functional.  
From Proposition \ref{secondclt} we readily deduce the stable convergence
\bea\label{180522-1}
\left(M^{n}, N^n, \sqrt{n}(C^n - C) \right)
\stab \left(M, N, \widehat{C} \right)
\eea
where $M^n=(M^{n,1}, \ldots, M^{n,k})$ and $N^n=(N^{n,1}, \ldots, N^{n,k})$. 
Now, we need to determine the mixed normal representation of the vector 
$ (M, N, \widehat{C})$. Note that the $\mathcal F$-conditional mean of the first and the third component is zero, which is due to   Proposition \ref{secondclt}.  
On the other hand, the $\mathcal F$-conditional mean of $N$ is not vanishing. Observing
the representation of $N$ in Proposition \ref{secondclt} and applying It\^o's formula we conclude
that 
\begin{align*}
\mu_j:= \overline{\E}[N^j| \mathcal F] = \int_0^{T_j} v_s^2 dW_s + A_{T_j}(3) + 
 \frac{1}{2} \int_0^{T_j}
\Sigma_{s} \left(a'' +b''-b'b'' \right)(X_s) \left(\int_0^s u_r^{11} dr \right) ds, 
\end{align*}
where the processes $v^2$, $u^{11}$ and $A(3)$ have been introduced in Propositions \ref{firstclt}
and  \ref{secondclt}. Furthermore, the $\mathcal F$-conditional covariance structure of
the vector $(M, N, \widehat{C})$ is fully determined by Proposition \ref{secondclt}. Thus,
setting $\mu=(\mu_1, \ldots, \mu_k)$, we may write 
\begin{align*}
 (M_1, N, \widehat{C}) \sim MN\left((0,\mu,0), 
 \left(
\begin{matrix}
\Theta_{11} & \Theta_{12} & \Theta_{13}  \\
\Theta_{21} & \Theta_{22} & \Theta_{13} \\
\Theta_{31} & \Theta_{32} & \Theta_{33}
\end{matrix}
\right)
  \right).
\end{align*}
Due to $\mathcal F$-conditional Gaussianity of the limit
$(M_1, N, \widehat{C}),$ the adaptive symbol $\underline{\sigma}$ has the following form: 
\begin{equation} \label{Lsigma}
\underline \sigma (z, \iu,\iv) = \frac{(\Theta_{31} \Theta_{11}^{-1} z)  [(\iu)^{\otimes 2}]}{2}+
(\mu + \Theta_{21} \Theta_{11}^{-1} z)[\iu]       \qquad z,u \in \R^k.
\end{equation}
 Combining two random symbols, we end up with the approximative density 
\begin{eqnarray}\label{20170404-1} &&
p_n^{(Z_n, G)}(z,y,x) = \E [ \phi(z;0,C)|  G=(y,x)] p^{ G}(y,x) 
\nn\\&& 
+n^{-1/2} \sum_j (-d_z)^{m_j} (-d_x)^{n_j}
\Big( \E \left[  c_j(z)\phi(z;0,C)| G=(y,x)\right] p^{ G}(y,x) \Big),
\end{eqnarray}
with $(z,x,y) \in \R^k \times \R^k \times \R^{q}$, as in \eqref{pn}. 

In this setting, however, 
the kernels $K^{n,j}$ and $K$ are defined by (\ref{settingw}), and 
the functionals $c_\alpha(z)$ in the representation (\ref{180520-2}) 
of the full random symbol $\sigma$ 
and also in (\ref{20170404-1}) 
are associated with 
$\underline{\sigma}$ of (\ref{Lsigma}) and 
$\overline{\sigma}$ of (\ref{170207-5}).

In the following we will assume the following condition: 
\bd
\im[(A)] The functions $a$ and $b$ are in $C^\infty(\bbR)$ and 
all their derivatives of positive order are bounded.  
\ed
Under $(A)$ conditions of Theorem \ref{20170206-1} can be slimed down. 
Recall that the variables $\bbI^j_s$ are defined by (\ref{180524-1}). 
\bd
\im[(C1)] For every $p>1$ and $j=1,...,k$, 
\beas
\sup_{s\in(T_{j-1},T_j)}
\left\|\left[\frac{\bbI^j_s}{T_j-s}\right]^{-1}\right\|_{{\mathbb L}^p}  &<& \infty.
\eeas
\ed

Recall that $\ell=k+\overline{q}+8$.  

\bd
\im[(C2)] $F\in\bbD_{\ell+1,\infty}(\bbR^{q})$, 
$\sup_{r_1,...,r_m\in(0,1)}\|D_{r_1,...,r_m}F\|_{\mathbb L^p} <\infty$ for every $p>1$ and 
$m=1,...,\ell+1$. Moreover, $r\mapsto D_r F$ and $(r,s)\mapsto D_{r,s}F$ ($r\leq s$) are continuous a.s. 
\ed

\bd
\im[(C3)]
$\det \sigma_G \in {\mathbb L}_{\infty-}$. 
\ed

Under the aforementioned conditions we obtain the following theorem, which is proved in Section \ref{180522-9}. 
\begin{theo}\label{maintheorem}
Suppose that conditions $(A)$, $(C1)$, $(C2)$ and $(C3)$ are fulfilled. 
Then, for every pair of positive numbers $(K,\gamma)$, 
\beas
\sup_{h\in \mathcal{E}(K,\gamma)}\left| \E[h(Z_n,G)] - \int_{\R^{k}
\times \R^{k} \times \R^{q}} h(z,y,x) p_n^{(Z_n,G)}(z,y,x) dzdydx \right|
&=&
o(n^{-1/2})
\eeas
as $n\to\infty$. 
\end{theo}

As a consequence of Theorem \ref{maintheorem} we finally obtain the approximative density 
of the pair $(V_n,  F)$ for $V_n=(V_{T_1}^n, \ldots, V_{T_k}^n)$ and an external $q$-dimensional random variable $F$.

\begin{cor} \label{cor4.4}
We set
\begin{align} \label{pnformula}
p^{(V_n, F)}_n (z,x)
= \int_{\R_+^k} \frac{1}{y_1 \cdots y_k} p_n^{(Z_n, G)} \left(z_1/y_1,
\ldots, z_k/y_k, y_1,\ldots, y_k, x  \right) dy.
\end{align}
Under the conditions of Theorem \ref{maintheorem} we obtain that 
\[
\sup_{h\in \mathcal{E}(K,\gamma)}\left| \E[h(V_n, F)] - \int_{\R^{k} 
\times \R^{ q}} h(z,x) p_n^{(V_n, F)}(z,x) dzdx \right|=
o(n^{-1/2}).
\]
\end{cor}

Theorem \ref{maintheorem} relies on 
the non-degeneracy of $G$. 
We will discuss 
some sufficient conditions  in the following subsections. 
When $k\geq2$, the non-degeneracy becomes a global problem 
and it is not so straightforward to consider the question 
in full generality. However, a localization method provides 
a practical solution.

\subsection{On condition $(C1)$}
In this section we will give a sufficient condition for $(C1)$. 
We are working in the setting of Section \ref{MultiEdgeworth} imposing assumption $(A)$. 
We consider the following condition: 

\bd
\im[ (C1$^\sharp$)] {\bf (i)} $\inf_{x\in\bbR}|b(x)|>0$. 
\bd
\im[(ii)] There exists a compact set $B \subseteq \bbR$ 
such that 
\bd\im[(a)] $\inf_{x\in B^c}| b'(x)|>0$,
\im[(b)] 
$\sum_{j=1}^\infty |b^{(j)}(x)|\not=0$ for each $x\in B$. 
\ed
\ed
\ed

For example, in the setting of null drift,  
if $X_t$ visits  the set $\{x:\>b'(x)=0\}$ after some time, 
then $\Sigma_t$ does not diffuse there 
and we never get non-degeneracy of $\Sigma_t$ thereafter. 
This explains the necessity of a global condition like $(C1^\sharp)$(ii)(a). 
As a matter of fact, such a degenerate case is essentially 
in the scope of the classical expansion 
for a martingale with an exactly normal limit (cf. \cite{Yoshida1997}). 
Now we have the following result. 
\begin{prop}
Condition $(C1)$ holds under $(A)$ and $(C1^\sharp)$. 
\end{prop}
\proof 
We need to show that 
\bea\label{20170129-1}
\sup_{s\in(T_{j-1},T_j)}
\left\|\left[\frac{\bbI^j_s}{T_j-s}\right]^{-1}\right\|_{\mathbb L^p} &<& \infty
\eea
for every $p>1$ and $j=1,\ldots,k$. 
%
%
Let $s\in(T_{j-1},T_j)$. Recalling \eqref{180524-1}, we have 
\beas 
\frac{\bbI^j_s}{T_j-s}
 &=& 
\frac{1}{2}\frac{1}{T_j-s}\int_s^{T_j}\Sigma_r^{-2}\{bb'(X_r)\}^2dr
\\&\geq&
\frac{1}{2} \inf_{r\in[s,T_j]}\Sigma_r^{-2}\times \frac{1}{T_j-s} \int_s^{T_j}\{bb'(X_r)\}^2dr. 
\eeas
%
By $(C1^\sharp)$ and the compactness of $B$, 
there exist a 
finite set $\caln\subset B$, 
a positive constant $c$ and 
an integer $m\geq2$ 
such that 
\bea\label{20181111-1} 
\{bb'(x) \}^2 &\geq& \min_{z\in\caln}c^{m/2}(1\wedge |x-z|^m)
\eea
for all $x\in\bbR$. 
{\colorr
Indeed, by $(C1^\sharp)$(i) and (ii)(a), there exists a positive constant $c'$ such that 
$\inf_{x\in B^c}\{bb'(x)\}^2\geq c'$. 
For each $z\in B$, by $(C1^\sharp)$(ii)(b), there exists an integer $j_z\geq1$ such that 
$b^{(j_z)}(z)\not=0$ and 
$b'(x)=((j_z-1)!)^{-1}b^{(j_z)}(z)(x-z)^{j_z-1}+\cdots$ for all $x$ near $z$. 
Therefore, from $(C1^\sharp)$(i), for each $z\in B$, 
there exists 
a positive constant $c_z$ and a neighborhood $B_z$ such that 
$
\{bb'(x)\}^2\geq c_z(1\wedge|x-z|^{m_z}) 
$ 
for all $x\in B_z$, 
with $m_z=(j_z-1)^2\geq0$. 
Since $B$ is compact, one can find a finite set $\caln\subset B$ such that 
$B\subset\cup_{z\in\caln}B_z$, and hence 
$$
\{bb'(x)\}^2
\geq\min_{z\in\caln}\big(\min_{z'\in\caln}c_{z'}\big)\big(1\wedge|x-z|^{\max_{z'\in\caln}m_{z'}}\big)$$ 
for all $x\in B$ 
since there exists $z$ for each $x\in B$ such that $x\in B_z$. 
If we set $c=\big(\min\{c',\min_{z\in\caln}c_z\}\big)^{2/m}$ for 
$m=\max\{2,\max_{z\in\caln}m_z\}$ we obtain (\ref{20181111-1}).

}
Let $\delta>0$ and $B_0:=\{x:~\text{dist }(x,\caln)<2\delta\}$.
Let $s_i=s+i(T_j-s)/n$. 
Then, there exists $n_0\in\bbN$ independent of $s$ such that for $n\geq n_0$, 
\begin{en-text}
\beas &&
\mathbb P\bigg[\frac{1}{T_j-s}\int_s^{T_j}\{bb'(X_r)\}^2dr\leq\frac{1}{n^{3m/2}}\bigg]
\\&\leq&
\sum_{z\in\caln}\mathbb P\bigg[\frac{c^{m/2}}{T_j-s}\int_s^{T_j}(1\wedge |X_r-z|^m)dr
\leq\frac{1}{n^{3m/2}}\bigg]
\\&\leq&
\sum_{z\in\caln}\mathbb P\bigg[\frac{c}{T_j-s}\int_s^{T_j}(1\wedge |X_r-z|^2)dr\leq\frac{1}{n^3}\bigg]
\\&\leq&
\sum_{z\in\caln}\sum_{i=1}^n \mathbb P\bigg[\frac{c}{T_j-s}\int^{s_i}_{s_{i-1}}(1\wedge |X_r-z|^2)dr\leq\frac{1}{n^4}\bigg]
\\&\leq&
\sum_{z\in\caln}\sum_{i=1}^n \mathbb P\bigg[\frac{c}{T_j-s}\int^{s_i}_{s_{i-1}}(1\wedge |X_r-z|^2)dr\leq\frac{1}{n^4}
,\>\inf_{r\in[s_{i-1},s_i]}|X_r-z|<n^{-1/2}\bigg]
\\&\leq&
\sum_{z\in\caln}\sum_{i=1}^n\mathbb P\bigg[\frac{c}{T_j-s}\int^{s_i}_{s_{i-1}}(1\wedge |X_r-z|^2)dr\leq\frac{1}{n^4}
,\>\sup_{r\in[s_{i-1},s_i]}|X_r-z|<n^{-1/3}\bigg]
+O(n^{-L})
\eeas
\end{en-text}
{\colorr
\beas &&
\mathbb P\bigg[\frac{1}{T_j-s}\int_s^{T_j}\{bb'(X_r)\}^2dr\leq\frac{1}{n^{3m/2}}\bigg]
\\&\leq&
\mathbb P\bigg[\frac{c^{m/2}}{T_j-s}\int_s^{T_j}\min_{z\in\caln}(1\wedge |X_r-z|^m)dr
\leq\frac{1}{n^{3m/2}}\bigg]
\\&\leq&
{\mathbb P}\bigg[\frac{c}{T_j-s}\int_s^{T_j}\min_{z\in\caln}(1\wedge |X_r-z|^2)dr\leq\frac{1}{n^3}\bigg]
\\&\leq&
\sum_{i=1}^n \mathbb P\bigg[\frac{c}{T_j-s}\int^{s_i}_{s_{i-1}}\min_{z\in\caln}(1\wedge |X_r-z|^2)dr\leq\frac{1}{n^4}\bigg]
\\&=&
\sum_{i=1}^n \mathbb P\bigg[\frac{c}{T_j-s}\int^{s_i}_{s_{i-1}}\min_{z\in\caln}(1\wedge |X_r-z|^2)dr\leq\frac{1}{n^4}
,\>\inf_{r\in[s_{i-1},s_i]}\min_{z\in\caln}|X_r-z|<n^{-1/2}\bigg]
\\&\leq&
\sum_{z\in\caln}\sum_{i=1}^n\mathbb P\bigg[\frac{c}{T_j-s}\int^{s_i}_{s_{i-1}}(1\wedge |X_r-z|^2)dr\leq\frac{1}{n^4}
,\>\sup_{u\in[s_{i-1},s_i]}|X_r-z|<n^{-1/3}\bigg]
+O(n^{-L})
\eeas
}\noindent
where $L$ is any positive number independent of $s$\noindent
{\colorr; in fact, on the event $\big\{\inf_{r\in[s_{i-1},s_i]}|X_r-z|<n^{-1/2}\big\}$ for $z\in\caln$, 
the process $X$ keeps $\sup_{r'\in[s_{i-1},s_i]}|X_{r'}-z|<n^{-1/3}$ with probability $1-O(n^{-L-1})$, 
and $\min_{z'\in\caln}|X_{r'}-z'|=|X_{r'}-z|$ for $n\geq n_0$ 
since the points in $\caln$ are isolated.} 
The first term of the right-hand side of {\colorr the above} inequality is bounded by 
\beas &&
\sum_{z\in\caln}\sum_{i=1}^n \mathbb P\bigg[\frac{c}{T_j-s}\int^{s_i}_{s_{i-1}}|X_r-z|^2dr\leq\frac{1}{n^4}
,\ X_r\in B_0\>\text{for all } r\in[s_{i-1},s_i]
\bigg]
\\&{\colorr=}&
\sum_{z\in\caln}\sum_{i=1}^n \mathbb P\bigg[\frac{1}{s_i-s_{i-1}}\int^{s_i}_{s_{i-1}}|X_r-z|^2dr
\leq\frac{1}{cn^3}
,\ X_r\in B_0\>\text{for all } r\in[s_{i-1},s_i]\bigg]
\eeas
{\colorr for large $n$.}
Since on the bounded set $B_0$, the process $X_r$ behaves like a Brownian motion, 
the last probability is bounded by 
${\colorr c_1^{-1}}n\exp(-c_1n)$ {\colorr for some positive constant $c_1$ 
independent of $s\in(T_{j-1},T_j)$, which follows from} a similar inequality to \cite[Lemma 10.6]{IkedaWatanabe1989}. 
Consequently, we obtain (\ref{20170129-1}) 
{\colorr by using the estimate 
\beas 
{\sup_{s\in(T_{j-1},T_j)}\mathbb E}[\Gamma_s^{-p}] &=& 
\sup_{s\in(T_{j-1},T_j)}\int_0^\infty pt^{p-1}\bbP[\Gamma_s<t^{-1}]dt
\\&\leq&
\sum_{n=0}^\infty p(n+1)^{3mp/2}\sup_{s\in(T_{j-1},T_j)}\bbP[\Gamma_s<n^{-3m/2}]
\><\>\infty
\eeas
for $\Gamma_s=(T_j-s)^{-1}\int_s^{T_j}\{bb'(X_r)\}^2dr$ and $p>1$.}
\qed

\subsection{On condition $(C3)$ for non-degeneracy of $G$ in the case $k=1$}\label{180527-3}

The problem of non-degeneracy of $\sigma_G$ can be reduced to local properties of the stochastic differential equations in the case $k=1$. 
Consider a system of stochastic differential equations in Stratonovich form
\bea\label{180527-2} 
d\bX_t&=& \bV_0(\bX_t)dt+\bV_1(\bX_t)\circ dW_t,
\qquad \bX_0\yeq(x_0,1,f)
\eea
for a $(2+q)$-dimensional process 
$\bX_t=(\bX^{(j)}_t)_{j=1,2,3}$, where 
$\bV_i=(\bV^{(j)}_i)_{j=1,2,3}$ $(i=0,1)$ are vector fields. 
The elements of $\bV_i$'s are  specified as follows: 
\beas 
\bV^{(1)}_0(\bx) &=& 
\tilde{a}(x_1)\>:=\> a(x_1)-\half b(x_1)b'(x_1),
\qquad
\bV^{(1)}_1(\bx) \>=\>b(x_1),\\
\bV^{(2)}_0(\bx) &=& 
\tilde{a}'(x_1)x_2\>=\> \bigg\{a'(x_1)-\half\big(b''(x_1)b(x_1)-(b'(x_1))^2\big)\bigg\}x_2,
\\
\bV^{(2)}_1(\bx) &=& b'(x_1)x_2.
\eeas
Suppose that the vector fields $\bV^{(3)}_i$ $(i=0,1)$ are smooth and their derivatives of positive order are bounded, and that 
the $q$-dimensional random variable $F$ is represented by the third element of $\bX_T$ as $F=\bX^{(3)}_T$, $T\in(0,1]$. 
In the case $F=0$, $\bX_t$ is $(\bX^{(1)}_t,\bX^{(2)}_t)$ 
and $\bV_i$ are 
$(\bV^{(1)}_i,\bV^{(2)}_i)$ $(i=0,1)$ respectively. 
By definition, $\bX^{(1)}_T=X_T$ and $\bX^{(2)}_T= \Sigma_T$. 

The Lie algebra generated by $\bV_i$ $(i=0,1)$ at $\overline{x}\in\bbR^{2+q}$ is denoted by $\text{Lie}[\bV_1,\bV_0](\overline{x})$, 
namely, it is the linear span of the vectors in $\cup_{i=0}^\infty\calv_i$ with 
$\calv_0=\{\bV_1(\bx)\}$, 
$\calv_i=\{[\bV_j(\bx),{\sf V}];\>{\sf V}\in\calv_{i-1}\}$ $(i\in\bbN)$, 
where $[V, W](x)=\mathcal{D} V(x) W(x)-\mathcal{D} W(x) V(x)$ with $\mathcal{D} V(x)$ being the derivative of $V$ at $x.$
A simple criterion for non-degeneracy of $\sigma_G$ is provided by the H\"ormander condition (see Section 2.3.2 in \cite{N} for details).

\begin{prop}\label{180527-1} 
Let $k=1$. For a constant $\bX_0$, if ${\rm span}\>{\rm Lie}[\bV_1,\bV_0](\bX_0)=\bbR^{2+q}$, then $(C3)$ holds. 
\end{prop}

A variation is the case where 
$F$ has a component $X_T$, that is, $F=(X_T,F_1)$; 
$F_1$ may be empty. 
If we have a representation $F_1=\bX^{(3)}_T$, then 
Proposition \ref{180527-1} remains valid.

The non-degeneracy problem for $\sigma_G$ becomes a global one when $k>1$ 
since we need non-degeneracy of $\Sigma_{T_2}-\Sigma_{T_1}$, but 
the support of $\Sigma_{T_1}$ is no longer compact. 
Though we could assume some strong condition that gives uniform non-degeneracy 
over the whole space, it would be a quite restrictive solution. 
Instead, in Section \ref{20180527-5}, we will consider a different way 
by slightly modifying Theorem \ref{maintheorem}, but 
such modification keeps the error bound of the approximation meaningful in practice.

\subsection{Localization}\label{20180527-5}
\begin{en-text}
Let us consider the situation of Section \ref{180527-3} with the system (\ref{180527-2}) of 
stochastic differential equations. 
For simplicity of presentation, we will only treat the following two extreme cases 
for a $k$-dimensional $F_1=(F_1^{(j)})_{j=1,...,k}$, while 
it is obvious to treat more general cases in the same manner. 

\noindent{\bf Case 1}. $F_1^{(j)}=X_{T_j}$ for $j=1,...,k$. 
In this situation, $\bX_t=(\bX^{(1)}_t,\bX^{(2)}_t)$ without the third component. 

\noindent{\bf Case 2}. $F_1^{(j)}=(\bX^{(3)}_{T_j})$ for $j=1,...,k$. 

In both cases, the variational process $\bSigma_t$ is defined by 
the stochastic differential equation 
\beas
d\bSigma_t &=& 
\partial_{\bx}\bV_0(\bX_t)\bSigma_tdt+\partial_{\bx}\bV_1(\bX_t)\bSigma_t\circ dW_t,
\qquad
\bSigma_0 \yeq I_*
\eeas
where $I_*$ is the unit matrix of order $\text{dim}(\bX_t)$. 
\end{en-text}
To convey the idea simply, we shall only treat the case 
$F=(X_{T_j})_{j=1,...,k}$, while more general cases can be formulated 
in a similar manner. 

Let us consider the situation of Section \ref{180527-3} with the system (\ref{180527-2}) of 
stochastic differential equations 
for $\bX_t=(\bX^{(1)}_t,\bX^{(2)}_t)=(X_t,\Sigma_t)$. 
\begin{en-text}
\bd
\im[(D1)] There is a finite closed interval $I$ in $\bbR$ such that $x_0\in I^o$ and that 
the following conditions hold. 
\bd
\im[(i)] $\inf_{x\in I}|b(x)|>0$. 
\im[(ii)] 
$\sum_{j=1}^\infty |b^{(j)}(x)|\not=0$ for each $x\in I$. 
\ed
\ed
\bd
\im[(D2)] ${\rm span}\>{\rm Lie}[\bV_1,\bV_0](x,1)=\bbR^2$ for $x\in I$. 
\ed
\end{en-text}
\bd
\im[(D)] $
{\rm Lie}[\bV_1,\bV_0](x,1)=\bbR^2$ for $x\in I$. 
\ed

For positive numbers $K$ and $\gamma$, 
let 
$\cale(K,\gamma,I)$
be the set of measurable functions 
$h:\bbR^{3k}\to\bbR$ 
such that $h(z,y,x)=0$ when 
$x_j\in I^c$ for some $j\in\{1,...,k-1\}$, 
$x=(x_j)_{j=1,...,k}$, and that 
$|h(z,y,x)|\leq M(1+|z|+|y|+|x|)^\gamma$ for all $(z,y,x)\in\bbR^{3k}$. 

Denote by $(X_t(s,x),\Sigma_t(s,(x,y)))$ 
the stochastic flow defined by 
\beas 
\left\{\begin{array}{ccl}
dX_t(s,x) &=& \tilde{a}(X_t(s,x))dt+b(X_t(s,x))\circ dW_t,\vspace{2mm}\\
d\Sigma_t(s,(x,y)) &=& \tilde{a}'(X_t(s,x))\Sigma_t(s,(x,y)) dt+b'(X_t(s,x))\Sigma_t(s,(x,y)) \circ dW_t
\end{array}\right.
\eeas
with $(X_s(s,x),\Sigma_s(s,(x,y)))=(x,y)$, $0\leq s\leq t\leq1$. 
Assume conditions $(A)$, $(C1)$ and $(D)$. 
Then by Proposition \ref{180527-1} and Theorem \ref{maintheorem}, 
for each $x_{j-1}\in I$ and $y_{j-1}>0$, 
there exists a density 
\beas  
q_n^{(j)}
\big(\zeta_j,\eta_j,x_j|y_{j-1},x_{j-1}\big)
&=&
p_n^{\big(V^n_{T_j}-V^n_{T_{j-1}},y_{j-1}^{-1}\Sigma_{T_j}(T_{j-1},(x_{j-1},y_{j-1})),
X_{T_j}(T_{j-1},x_{j-1})\big)}
\big(\zeta_j,\eta_j,x_j\big)
\eeas
with initial value $(\Sigma_{T_{j-1}},X_{T_{j-1}})=(y_{j-1},x_{j-1})$ 
of the system starting at time $T_{j-1}$ 
that gives the asymptotic expansion 
\beas&&
\E\big[h_j\big(V^n_{T_j}-V^n_{T_{j-1}},y_{j-1}^{-1}\Sigma_{T_j},X_{T_j}\big)
\big| \Sigma_{T_{j-1}}=y_{j-1},\>
X_{T_{j-1}}=x_{j-1}\big]
\\&&
-\int_{\bbR^3}h_j(\zeta_j,\eta_j,x_j)q_n^{(j)}
\big(\zeta_j,\eta_j,x_j|y_{j-1},x_{j-1}\big)
d\zeta_jd\eta_jdx_j
\\&=& 
o(n^{-1/2})
\eeas
uniformly in $h_j\in\cale(K,\gamma)$ for every $(K,\gamma)\in(0,\infty)^2$. 
Indeed, $q_n^{(j)}\big(\zeta_j,\eta_j,x_j|y_{j-1},x_{j-1}\big)$ is 
the density $p_n(\zeta_j,\eta_j,x_j)$ in the one-step case starting from time $T_{j-1}$ 
and 
the initial values $X_0=x_{j-1}\in I$ and $\Sigma_0=1$. 
Then we obtain a function $q_n^{(Z_n,G)}(z,y,x)$ that approximates the distribution of $(Z_n,G)$ with 
$G=((\Sigma_{T_j})_{j=1,...,k},(X_{T_j})_{j=1,...,k})$: 
\beas 
q_n^{(Z_n,G)}(z,y,x)
&=& 
\prod_{j=1}^k q_n^{(j)}\big(z_j-z_{j-1},y_{j-1}^{-1}y_j,x_j\big|y_{j-1},x_{j-1}\big)y_{j-1}^{-1}
\eeas
for $(z,y,x)=\big((z_j)_{j=1,...,k},(y_j)_{j=1,...,k},(x_j)_{j=1,...,k}\big)$, 
$(z_0,y_0)=(0,1)$. 
We should remark that this function is defined 
only when $x_{j-1}\in I$ for $j=1,...,k$. 
Now we give a localized version of Theorem \ref{maintheorem}. 
\begin{theo}\label{maintheorem2}
Suppose that Conditions $(A)$, $(C1)$ and $(D)$ are fulfilled for some finite closed interval $I$. 
Let $G=((\Sigma_{T_j})_{j=1,...,k},(X_{T_j})_{j=1,...,k})$. 
Then, for every pair of positive numbers $(K,\gamma)$, 
\beas
\sup_{h\in \mathcal{E}
(K,\gamma,I)}\left| \E[h(Z_n,G)] - \int_{\R^{k}
\times \R^{k} \times \R^{k}} h(z,y,x) q_n^{(Z_n,G)}(z,y,x) dzdydx \right|
&=&
o(n^{-1/2})
\eeas
as $n\to\infty$. 
\end{theo}


For a sketch of the proof of Theorem \ref{maintheorem2}, we notice that 
the function $h$ admits the estimate 
\beas 
|h(z,y,x)| &\leq& 
M_1\prod_{j=1}^k\big(1+|z_j|+|y_j|+|x_j|)^{\gamma_1}
\eeas
for some $(M_1,\gamma_1)\in(0,\infty)^2$. 
Then repeated use of the approximation yields the desired error bound. 

The asymptotic expansion for $(V_n,(X_{T_j})_{j=1,...,k})$ as in Corollary \ref{cor4.4} 
also follows under  conditions of Theorem \ref{maintheorem2}. 


\begin{en-text}
\koko
The variational process $\bSigma_t$ is defined by 
the stochastic differential equation 
\beas
d\bSigma_t &=& 
\partial_{\bx}\bV_0(\bX_t)\bSigma_tdt+\partial_{\bx}\bV_1(\bX_t)\bSigma_t\circ dW_t,
\qquad
\bSigma_0 \yeq I_{2}
\eeas
where $I_{2}$ is the unit matrix of order $2$.

\beas 
\bV_0(x_1,x_2,x_3)&=&
\left[
\begin{array}{c} \tilde{a}(x_1)
\end{array}
\right]
\eeas

\end{en-text}
\begin{en-text}
{\colorr 
\begin{remark}\rm to the authors. 
What we obtained is only in the case $F'=\emptyset$ i.e. $q=0$. 
``$F'$'' is used in the following theorem and the corollary. 
\end{remark}
}
\end{en-text}

\section{Applications} \label{sec6}
\setcounter{equation}{0}
\renewcommand{\theequation}{\thesection.\arabic{equation}}

\subsection{Strong and weak error expansions} \label{sec6.1}

As the first application of the density expansion introduced in  \eqref{pnformula} we study the strong 
and the weak approximation error associated with the Euler approximation scheme.

\begin{prop} \label{prop5.1}
(Weak and strong approximation errors) 
Suppose that  conditions of Theorem \ref{maintheorem} are satisfied.
\begin{itemize} 
\item[\textnormal{(i)}] (Strong approximation error) Let $p_n^{V_n}(z)$ be the marginal density obtained from $p^{(V_n,F)}_n (z,x)$,
defined at \eqref{pnformula}, by projection
onto the first component and let $U_n=(X_{T_1}^n,\ldots, X_{T_k}^n)-(X_{T_1},\ldots, X_{T_k}).$ 
Then we obtain the following expansion for the $L^p$-norm of the approximation error
\[
\E[\| U_n \|^p]^{1/p} = n^{-1/2} \left(\int_{\R^k} \|z\|^p p_n^{V_n}(z) dz \right)^{1/p} + o(n^{-1/2}).
\]
\item[\textnormal{(ii)}] (Weak approximation error)
Consider a function $f \in C^2(\R^k)$ such that the second derivative of $f$ has polynomial growth.
  Setting $p^{(V_n,F)}_n (z,x)= p_1 (z,x) + n^{-1/2} p_2 (z,x)$
we deduce the asymptotic expansion
\begin{align*}
&\E[f(X_{T_1}^n,\ldots, X_{T_k}^n) - f(X_{T_1},\ldots, X_{T_k})] \\[1.5 ex]
&=n^{-1} \int_{\R^k \times \R^k} \left(\langle \nabla f (x), z \rangle
\cdot p_2 (z,x) + \frac 12 z^{\star}\text{Hess}f(x) z \cdot p_1 (z,x) \right) 
dzdx + o(n^{-1}).
\end{align*} 
\end{itemize}
\end{prop}

\begin{proof} 
Part (i) of the statement is a direct consequence of Corollary \ref{cor4.4} applied to the function 
$h(z)=\|z\|^p$. Now, we set $\bold{X}^n=(X_{T_1}^n,\ldots, X_{T_k}^n)$ and 
$\bold{X}=(X_{T_1},\ldots, X_{T_k})$. 
To obtain part (ii) of Proposition \ref{prop5.1} we apply Taylor expansion to conclude
that
\begin{align*}
f(\bold{X}^n) - f(\bold{X}) &=  \langle \nabla f (\bold{X}), \bold{X}^n - \bold{X} \rangle 
+ \frac{1}{2} (\bold{X}^n - \bold{X})^{\star}\text{Hess}f(\bold{X}) (\bold{X}^n - \bold{X}) \\
& +  \frac{1}{2} (\bold{X}^n - \bold{X})^{\star} \left(\text{Hess}f(\bold{Y}^n) - \text{Hess}f(\bold{X})
\right) (\bold{X}^n - \bold{X}),
\end{align*}
for some random vector $\bold{Y}^n \in \R^k$ with $\| \bold{Y}^n - \bold{X}\| 
\leq \|\bold{X}^n - \bold{X}\|$. In particular, $\bold{Y}^n \toop \bold{X}$. 
Observe that 
\[
\E\left[(\bold{X}^n - \bold{X})^{\star} \left(\text{Hess}f(\bold{Y}^n) - \text{Hess}f(\bold{X})
\right) (\bold{X}^n - \bold{X})\right] \to 0 \qquad \text{as } n\to \infty,
\]
which is due to $f \in C^2(\R^k)$.
We deduce the expansion 
\begin{align*}
&\E[f(X_{T_1}^n,\ldots, X_{T_k}^n) - f(X_{T_1},\ldots, X_{T_k})] \\[1.5 ex]
&=n^{-1} \int_{\R^k \times \R^k} \left(\langle \nabla f (x), z \rangle
\cdot p_2 (z,x) + \frac 12 z^{\star}\text{Hess}f(x) z \cdot p_1 (z,x) \right) 
dzdx + o(n^{-1})
\end{align*} 
since, according to Theorem \ref{maintheorem} and Corollary \ref{cor4.4} applied to 
$F=(X_{T_1}, \ldots, X_{T_k})$, it holds that 
\[
\int_{\R^k \times \R^k} \langle \nabla f (x), z \rangle
\cdot p_1 (z,x) dzdx = 0,
\]
because the $dz$-integral is taking over an odd function in $z$. This completes the proof 
of Proposition \ref{prop5.1}. 
\end{proof}

We remark that the weak error expansion of Proposition \ref{prop5.1}(ii) has been obtained in
\cite{BT1, BT2} for $k=1$ and the discrete Euler scheme. Furthermore, the authors proved that 
the error of the expansion in Proposition \ref{prop5.1}(ii) is $O(n^{-2})$, which is more precise than
$o(n^{-1})$.   We note however that the theory developed in \cite{BT1, BT2} is not sufficient to obtain 
the density expansion  \eqref{pnformula} of Corollary \ref{cor4.4}.

\subsection{Studentized statistics} \label{sec6.2}
In this part  we will apply results of Section \ref{MultiEdgeworth} to derive the density of the studentized statistic. To avoid complex notations, we restrict our attention to the case $k=1$. 

To this end, let $T \in [0,1].$ We note that $V_T^n=\Sigma_T Z_T^n$ and $V_T \sim MN(0, S_T)$ with $S_T= \Sigma_T^2 C_T.$
Then, the studentized statistic is 
\begin{equation}
\frac{V_T^n}{\sqrt{S_T}}=\frac{Z_T^n}{\sqrt{C_T}}
\end{equation}
Hence, it suffices to derive the density of the studentized statistic $Z_T^n/\sqrt{C_T}.$

We write \eqref{Lsigma} in the form
\begin{equation*}
\underline \sigma (z, \iu,\iv) = \mathcal{H}_1 z (\iu)^2+(\mathcal{H}_2+\mathcal{H}_3 z) (\iu)
\end{equation*}
Moreover, we restructure \eqref{170207-5} 
as
\begin{align} \label{Usigma}
 \overline{\sigma}(\iu,\iv)=& \mathcal{H}_4 (\iu)^3 + \mathcal{H}_5 (\iu) (\iv) 
+\mathcal{H}_6 (\iu)^5 + \mathcal{H}_7(\iu)(\iv)^2+ \mathcal{H}_8(\iu)^3 (\iv).
\end{align}
Adding these two random symbols, we obtain the full random symbol
\begin{align}\label{SumRandomSymbols}
\sigma (z, \iu,\iv)=\sum_{j=1}^7 c_j(z) (\iu)^{m_j} (\iv)^{n_j}.
\end{align}
where the components of $(m,n)=\left((m_j, n_j) \right)_{1 \leq j \leq 7}$ and 
$c(z)=(c_j(z))_{1 \leq j \leq 7}$ are given by
$$(m,n)=\left( (1,0 ), (2,0), (1,1), (3,0), (1,2), (3,1), (5,0)   \right)$$
and
$$c(z)=\left( \mathcal{H}_2+\mathcal{H}_3 z,  \mathcal{H}_1 z, \mathcal{H}_5, \mathcal{H}_4, \mathcal{H}_7, \mathcal{H}_8, \mathcal{H}_6  \right).$$
In view of Theorem \ref{maintheorem} and denoting $C=C_T, Z_n=Z_T^n,$  we obtain that
$$p_n^{(Z_n,C)}(z,x)=\phi(z;0,x) p^{C}(x)+n^{-1/2} \sum_{j=1}^{8} p_j(z,x)$$
where for each $j$ we have
$$p_j(z, x)=(-d_z)^{m_j} (-d_x)^{n_j}
\Big( \phi(z;0,x) p^{C}(x) \E \left[  c_j(z) |C=x\right]\Big).$$
Note that, in this case, most of the terms are the same as in \cite[Section 6]{PY}. Hence, adopting their derivations, we easily obtain the following identities: 
\begin{align*}
\int_{\R^2} g(z/\sqrt{x}) p_1(z, x) dz dx&=\E[\mathcal{H}_2 C^{-1/2}] \int_{\R} g(y) y \phi(y;0,1) dy \\
     &+\E[\mathcal{H}_3] \int_{\R} g(y) (y^2-1) \phi(y;0,1) dy,   \\
\int_{\R^2} g(z/\sqrt{x}) p_2(z, x) dz dx&=\E[\mathcal{H}_1 C^{-1/2}] \int_{\R} g(y) H_3(y) \phi(y;0,1) dy,   \\
\int_{\R^2} g(z/\sqrt{x}) p_3(z, x) dz dx&=\frac{-1}{2} \E[\mathcal{H}_5 C^{-3/2}] \int_{\R} g(y) (y^3-2y) \phi(y;0,1) dy,  \\
\int_{\R^2} g(z/\sqrt{x}) p_4(z, x) dz dx&=\E[\mathcal{H}_4 C^{-3/2}] \int_{\R} g(y) H_3(y) \phi(y;0,1) dy,  \\
\int_{\R^2} g(z/\sqrt{x}) p_5(z, x) dz dx&=\frac{1}{4} \E[\mathcal{H}_7 C^{-5/2}] \int_{\R} g(y) (y^5-4y^3) \phi(y;0,1) dy  ,   \\
\int_{\R^2} g(z/\sqrt{x}) p_6(z, x) dz dx&=\frac{-1}{2} \E[\mathcal{H}_8 C^{-5/2}] \int_{\R} g(y) (y^5-7y^3+6 y) \phi(y;0,1) dy  ,\\
\int_{\R^2} g(z/\sqrt{x}) p_7(z, x) dz dx&=\E[\mathcal{H}_6 C^{-5/2}] \int_{\R} g(y) H_5(y) \phi(y;0,1) dy,   \\
\end{align*}
where
$$H_3(y)=y^3- 3y, \quad H_5(y)=y^5-10 y^3+15 y.$$
Due to $F=C$ in $\sigma_t(\iu,\iv)$ of \eqref{170207-5}, we notice the equalities 
$$\mathcal{H}_4=\frac{\mathcal{H}_5}{2} \mbox{ and }  4 \mathcal{H}_6=\mathcal{H}_7=  \mathcal{H}_8,$$
which leads to the following result.

\begin{cor} Under conditions of Theorem \ref{maintheorem}, the Edgeworth expansion is
\begin{equation}\label{EE1dim}
p^{Z_n/\sqrt{C}}(y)=\phi(y;0, 1)+n^{-1/2} \phi(y;0,1)( a_1 y+ a_2 (y^2-1)+a _3 y^3)
\end{equation}
where
\begin{align*}
a_1&= \E[\mathcal{H}_2 C^{-1/2}]-3 \E[\mathcal{H}_1 C^{-1/2}]+\frac{1}{2} \E[\mathcal{H}_5 C^{-3/2}]+3 \E[\mathcal{H}_6 C^{-5/2}],\\
a_2&=\E[\mathcal{H}_3], \\
a_3&=\E[\mathcal{H}_1 C^{-1/2}].
\end{align*}
\end{cor}

\section{Proofs} \label{sec7}
\setcounter{equation}{0}
\renewcommand{\theequation}{\thesection.\arabic{equation}}

Throughout this section all positive constants are denoted by $C$ although they may change from line 
to line. Furthermore, due to a standard localisation procedure (see e.g. \cite{BGJPS06}) all continuous
stochastic processes $(Y_t)_{t \in [0,1]}$ can be assumed to be uniformly bounded in $(\omega,t)$ when proving  Theorems \ref{th2n} and \ref{th3}. In particular, it applies to stochastic processes 
$Y_t = a^{(l)}(X_t)$ and $Y_t = b^{(l)}(X_t)$ for $l=0,1,2$. For a generic diffusion process
$(Y_t)_{t \in [0,1]}$ of the form  \eqref{sde} with bounded coefficients we obtain the inequality
\begin{align} \label{BDG}
\E[|Y_t-Y_s|^p] \leq C_p |t-s|^{p/2} \qquad \text{for any } p>0 \text{ and } t,s \in[0,1], 
\end{align}
which holds due to Burkholder-Davis-Gundy  inequality. We will use the notation $Y^n \ucp Y$
to denote the uniform convergence in probability $\sup_{t \in [0,1]}|Y^n_t - Y_t| \toop 0$. In the proofs
we will  deal with sequences of stochastic processes of the form
\[
Y^n_t = \sum_{i=1}^{[nt]} \xi_i^n,
\] 
where $\xi_i^n$, $i=1,\ldots,n$, are $\mathcal F_{i/n}$-measurable random variables with 
$\E[|\xi_i^n|^p]<\infty$ for any $p>0$. The following statements trivially hold:
\begin{align} \label{negligibility} 
 &\sum_{i=1}^{[nt]} \E[|\xi_i^n|] \to 0   \qquad \Rightarrow \qquad  Y^n \ucp 0 \\
&  \sum_{i=1}^{[nt]} \E[\xi_i^n|~\mathcal F_{(i-1)/n}] \ucp Y_t  \quad \text{ and }  
\quad  \sum_{i=1}^{[nt]} \E[(\xi_i^n)^2|~\mathcal F_{(i-1)/n}] \toop 0  \nonumber \\
& \label{doob} \Rightarrow Y^n \ucp Y.  
\end{align}

\subsection{Proof of Theorem \ref{20170206-1}}
We will sketch the proof, basically following the ideas of \cite[Theorem 4]{Yoshida2013}, but outlining the difference 
caused by the multiple stopping in the present situation. Note that as in \cite[Theorem 4]{Yoshida2013}, it suffices to verify assumptions of 
\cite[Theorem 1]{Yoshida2013}.

\subsubsection{Construction of the truncation functional $\psi_n$ from $s_n$ and other variables} 
Let $\bar{d}$ satisfy the inequality $1/3<\bar{d}<d<1/2$, where the constant $d$ has been introduced before assumption (B1), and 
define $\xi_n$ by 
\beas 
\xi_n &=& 
10^{-1}n^{{\colorr 2\bar{d}}}|C_n-C|^2+2\big[1+4\det\sigma_{(M_n,F)}(s_n^k)^{-1}\big]^{-1}
\\&&
+
\int_{[0,1]^2}
\left(\frac{|C^n_t-C_t-C^n_s+C_s|n^{\bar{d}}}{|t-s|^{3/8}}\right)^8dtds. 
\eeas
\begin{en-text}
where 
$\sfc$ is a constant satisfying $2/3<1-\sfa<\sfc<1$. 
\end{en-text}
We define $Q_n=(M_n,F)$, $R_n=(N_n,\widehat{F}_n)$ and set
\beas 
R_n' &=& \sigma_{Q_n}^{-1}\bigg(n^{-1/2}\langle DQ_n,DR_n\rangle_{\mathbb H}
+n^{-1/2}\langle DR_n,DQ_n\rangle_{\mathbb H} +n^{-1}\langle DR_n,DR_n\rangle_{\mathbb H}\bigg). 
\eeas
Let $\psi\in C^\infty(\bbR;[0,1])$ be a function such that $\psi(x)=1$ if $|x|\leq1/2$ and 
$\psi(x)=0$ if $|x|\geq1$. 
We introduce the random truncation
\beas 
\psi_n &=& \psi(\xi_n)\psi(n^{1/2}|R_n'|^2).
\eeas
Remark that $\psi_n$ is well defined because so is 
$\sigma_{Q_n}^{-1}$ under the truncation by $\xi_n$. 
{\colorr In fact, if $\xi_n\leq1$, then 
$\det\sigma_{Q_n}\geq s_n^k/4$, that is nondegenerate thanks to $(B3)$(ii). 
Therefore $\sigma_{Q_n}^{-1}$ makes sense on the event $\{\xi_n\leq1\}$. 
We are defining $\psi_n=0$ on the event $\{\xi_n>1\}$ 
since $\psi(\xi_n)=0$ there. Thus, $\psi_n$ is well-defined.}

\subsubsection{Characteristic function and its decomposition}
Let $\check{Z}_n=(Z_n,F_n)$ and let $\check{Z}_n^\alpha=Z_n^{\alpha_1}F_n^{\alpha_2}$ 
for $\alpha=(\alpha_1,\alpha_2)\in\bbZ_+^k\times\bbZ_+^{q}$. 
Define
\beas 
\hat{g}_n^\alpha(u,v) &=& \mathbb E\big[\psi_n\check{Z}_n^\alpha\exp\big(Z_n[\iu]+F_n[\iv]\big)\big]
\eeas
for $u\in\bbR^k$ and $v\in\bbR^{q}$ and let 
\bea\label{170207-1}
g_n^\alpha(z,x) &=& (2\pi)^{- (k+q)}\int_{\bbR^{ k+q}} 
\exp\big(-z[\iu]-x[\iv]\big)\hat{g}_n^\alpha(u,v)dudv. 
\eea
The existence of the integral (\ref{170207-1}) can be verified by 
the nondegeneracy of the Malliavin covariance matrix of $(Z_n,F_n)$ 
under the truncation by $\psi_n$. 
We define the quantities 
\begin{align*}
\Psi(u,v) &= 
\exp\bigg(-\half C[u^{\otimes2}]+{\tt i} F[v]\bigg), \\[1.5 ex]
\ep_n(u,v) &= -\half(C_n-C)[u^{\otimes2}]+{\tt i}(F_n[v]-F[v])+{\tt i} n^{-1/2}N_n[u], \\[1.5 ex]
e^n_t(u) &= 
\exp\bigg({\tt i} M^n_t[u]+\half C^n_t[u^{\otimes2}]\bigg), \\[1.5 ex]
L^n_t(u) &= e^n_t(u)-1 \qquad \text{and} \qquad \dote(x) = \int_0^1 e^{sx}ds.
\end{align*}
Finally, we introduce the functions
\begin{align*}
\Phi^{1,\alpha}_n(u,v) &=
\partial^\alpha \mathbb E\bigg[e^n_1(u)\Psi(u,v)\ep_n(u,v)\dote(\ep_n(u,v))\psi_n\bigg], \\[1.5 ex]
\Phi^{2,\alpha}_n(u,v) &=
\partial^\alpha  \mathbb E\bigg[L^n_1(u)\Psi(u,v)\psi_n\bigg].
\end{align*}
The existence of $\Phi^{1,\alpha}_n(u,v)$ and $\Phi^{2,\alpha}_n(u,v)$ 
involving $e^n_1(u)\Psi(u,v)$ is ensured by the truncation $\psi_n$. 
Let us set
\beas 
\Phi_n^{0,\alpha}(u,v) &=& 
\partial^\alpha  \mathbb  E\big[\Psi(u,v)\psi_n\big]. 
\eeas
Then $\hat{g}_n^\alpha(u,v)$ possess the decomposition
\beas 
\hat{g}_n^\alpha(u,v) &=& \Phi_n^{0,\alpha}(u,v)+\Phi_n^{1,\alpha}(u,v)+\Phi_n^{2,\alpha}(u,v).
\eeas

\subsubsection{Error bound}
\begin{en-text}
Let 
\beas 
h^0_n(z,x) &=& 
E\big[\psi_n\phi(z;0,C)|F=x\big]p^F(x)
\\&&
+n^{-1/2} \sum_\alpha (-\partial_z)^{\alpha_1}(-\partial_x)^{\alpha_2}
\bigg\{E\big[c_\alpha(z)\phi(z;0,C)|F=x\big]p^F(x)\bigg\}. 
\eeas
We follow the argument in Section 6.2 of \cite{Yoshida2013} to obtain (42) therein, 
except for the parts concerning the derivation of 
(38) and (43) there based on the non-degeneracy argument. 
By this, we can identify $\overline{\sigma}$ as in (\ref{170207-5}). 
Then Theorem 1 in Section 4 
of \cite{Yoshida2013} gives the estimate
\end{en-text}
\begin{en-text}
\bea\label{170207-6} 
\sup_{(z,x)\in\bbR^{d+q}}
\big|(|z|+|x|)^m
\big(g^0_n(z,x)-h^0_n(z,x)\big)\big|
&=& 
o(n^{-1/2})
\eea
for every $m\in\bbZ_+$ 
\end{en-text}
We apply \cite[Theorem 1]{Yoshida2013} by verifying 
conditions $[B1]$, $[B2]_\ell$, $[B3]$ and $[B4]_{\ell,{\mathfrak m},{\mathfrak n}}$ therein under our assumptions (B1), (B2), (B3).
Remark that ``$\ell$'' therein
corresponds to $\check{\sfd}+3$, where $\check{\sfd}=k+q$.
Condition $[B1]$ follows from (B2)(iii) and a standard central limit theorem with a mixed normal limit. 
Condition $[B2]_\ell$ is verified by (B2)(i), (B1)(i)-(ii), (B2)(ii) and the definition of $\xi_n$. 

Condition $[B3]$ {\colorr is verified as follows. 
$C^{n,j}$ and $C^j$ are expressed as 
\beas 
C^{n,j}_t &=& 
n\sum_{i=1}^{m_n}\int_{i=1}^{m_n}K^{n,j}(t_{j-1})^2\int_{t_{i-1}\wedge t}^{t_i\wedge t}
\bigg(\int_{t_{i-1}}^sdW_r\bigg)^2ds
\eeas
and 
\beas 
C^j_t &=& \half\int_0^{t\wedge T_j}K(s)^2ds.
\eeas
Routinely, we have 
\beas 
\sup_{n\in\bbN}\bigg\|n^{\bar{d}}\sup_{t\in[0,1]}\big|C^{n,j}_t-C^j_t\big|\bigg\|_p
&<&
\infty
\eeas
for every $p>1$ from $(B1)$(i) and (ii). 
Therefore, $[B3]$(i) follows as 
\beas 
\bbP[|\xi_n|>1/2]
&\leq&
\bbP[n^{\bar{d}}|C^n_1-C_1|\geq 1]+\bbP[\det\sigma_{\bbX^k_1}\leq s_n^k]
\\&&
+\bbP\bigg[\int_{[0,1]^2}
\left(\frac{|C^n_t-C_t-C^n_s+C_s|n^{\bar{d}}}{|t-s|^{3/8}}\right)^8dtds\geq\frac{1}{10}\bigg]
\\&\to&
0
\eeas
 as $n\to\infty$ thanks to $(B3)$(i) and $(B1)$(ii). 
 By the definition of $\xi_n$, on the event $\{|\xi_n|<1\}$, 
 $n^{\sf (1-{\sf a})/2}|C_n-C|\leq 1$ for large $n$, which is $[B3]$(ii). 
 Moreover, $[B3]$(iii) follows from $(B3)$(ii) since 
 $\limsup_{n\to\infty}\bbE\big[1_{\{|\xi_n|\leq1\}}\det\sigma_{\bbX^k_1}^{-p}\big]
 \leq \limsup_{n\to\infty}\bbE\big[4^p(s_n^k)^{-p}\big]<\infty
 $.

}
Condition $[B4]_{\ell,{\mathfrak m},{\mathfrak n}}$(i) is rephrased as (B2)(iv). 
The present $\overline{\sigma}$ is in $\cals(\check{\sfd}+3,5,2)$ 
in particular; see \cite[p. 892]{Yoshida2013} for the relevant definitions. 
Thus, \cite[Theorem 1]{Yoshida2013} gives the error bound 
\beas
\sup_{f\in\cale(R,\gamma)}\Delta_n(f) 
&=& 
o(n^{-1/2})
\eeas
if the following two conditions are fulfilled: 
\bea\label{170208-1}
\lim_{n\to\infty}n^{1/2}\Phi^{2,\alpha}_n(u,v) &=& 
\partial^\alpha  \mathbb E\big[\Psi(u,v)\overline{\sigma}(\iu,\iv)\big]
\eea
for $u\in\bbR^k$, $v\in\bbR^{q}$ and $\alpha\in\bbZ_+^{\check{\sfd}}$, 
and 
\bea\label{170208-2}
\sup_n\sup_{(u,v)\in\Lambda^0_n(\check{\sfd},\sfq)}
n^{1/2}|(u,v)|^{\check{\sfd}+1-\ep}\big|\Phi^{2,\alpha}_n(u,v)\big| &<& \infty
\eea
for some $\ep=\ep(\alpha)\in(0,1)$ for every $\alpha\in\bbZ_+^{\check{\sfd}}$, 
where 
$\Lambda^0_n(\check{\sfd},\bar{d})=\{(u,v)\in\bbR^{\check{\sfd}};\>|(u,v)|\leq n^{\bar{d}/2}\}$. 

We obtain (\ref{170208-1}) as in  \cite[Eq. (41)]{Yoshida2013}, 
except for the parts concerning the derivation of  \cite[Eqs. (38) and (43)]{Yoshida2013}
by a non-degeneracy argument. 
We shall show (\ref{170208-2}). 
By using duality twice for the double stochastic integrals, we have 
\beas&& 
n^{1/2}\Phi^{2,\alpha}_n(u,v) 
\\&=& 
n\sum_{\sfa_0,\sfa_1:\sfa_0+\sfa_1=\alpha}c_{\sfa_0,\sfa_1}
\sum_{i=1}^{m_n}
\int_{t_{i-1}}^{t_i}\int_r^{t_i}  \mathbb E\bigg[\partial^{\sfa_0}K^n(t_{i-1})[\tti u]
\partial^{\sfa_1}D_r \bigg\{e^n_s(u)D_s\big(\Psi(u,v)\psi_n\big)\bigg\}\bigg]dsdr
\eeas
for some constants $c_{\sfa_0,\sfa_1}$. 
\begin{en-text}
\beas &&
D_{r_1,...,r_m}
\partial^{\sfa_0}K^n(t_{i-1})[\tti u]\partial^{\sfa_1}D_r \bigg\{e^n_s(u)D_s\big(\Psi(u,v)\psi_n\big)\bigg\}
\\&=&
e^n_s(u)\Psi(u,v)\sigma(n,\sfa_0,\sfa_1,r,s,r_1,...,r_m;\tti u,\tti v),
\eeas
for $m=0,1,...,\ell-2$, 
where $\sigma(n,\sfa_0,\sfa_1,r,s,r_1,...,r_m;\tti u,\tti v)$ is a polynomial random symbol that admits 
\beas 
\big|D_{r_1,...,r_m}\sigma(n,\sfa_0,\sfa_1,r,s,r_1,...,r_m;\tti u,\tti v)\big| &\leq& 
(1+|(u,v)|)^5A(r,s,r_1,...,r_m)
\eeas
on $(u,v)\in\Lambda^0_n(\check{\sfd},\sfq)$ under truncation by $\psi_n$ 
with some random variables $A(r,s,r_1,...,r_m)$ such that 
$\{A(r,s,r_1,...,r_m);r,s,r_1,...,r_m\}$ is bounded in $L_{\infty-}$; for example, 
factors like $(D_{r_1}(C^n_t-C_n)[u^{\otimes2}]$ and $(C^n_t-C^n_1)[u,\cdot]$ come out 
\end{en-text}
We have 
\beas 
D_r \bigg\{e^n_s(u)D_s\big(\Psi(u,v)\psi_n\big)\bigg\}
&=&
e^n_s(u)\Psi(u,v)\sigma(n,r,s;\tti u,\tti v)
\\&=&
\bbF^n_s\bbG_s\bbH^n_s
\sigma(n,r,s;\tti u,\tti v),
\eeas
where 
\beas 
\bbF^n_s &=&  \exp\bigg(M^n_s[\tti u]+F[\tti v]\bigg),
\\
\bbG_s &=& \exp\bigg(-\half\big(C_1-C_s)[u^{\otimes2}]\bigg),
\\
\bbH^n_s &=& \exp\bigg(\half \big(C^n_s-C_s\big)[u^{\otimes2}]\bigg) 
\eeas
and 
$\sigma(n,r,s;\tti u,\tti v)$ is  a polynomial random symbol 
of fourth order in $(u,v)$ with coefficients in $\bbD_{\ell-2,\infty}(\bbR)$. 

First, we will consider the case $\alpha=0$, 
and estimate $n^{1/2}\Phi^{2,0}_n(u,v)$. 
Let $s\in(T_{j-1},T_j)$. 
Then 
$M^n_s\>=\>\big(M^{n,1}_{T_1},...,M^{n,j-1}_{T_{j-1}},M^{n,j}_s,
\ldots , M^{n,k}_s\big)
$.
We will estimate the speed of the decay of 
the expectations of the components of 
$n^{1/2}\Phi^{2,\alpha}_n(u,v)$ for $(u,v)\in\Lambda^0_n(\check{\sfd},\bar{d})$. 
Our strategy is as follows. 
For $s\in(\tau^j_n,T_j)$, we apply the integration-by-parts formula 
for 
$
\big(M^{n,1}_{T_1},...,M^{n,j-1}_{T_{j-1}},M^{n,j}_s,F\big)
$
to obtain the decay $|(u_1,...,u_j,v_1,...,v_{q})|^{-(\check{\sfd}+1-\ep)}$, 
where $u=(u_1,...,u_k)$ and $v=(v_1,...,v_q)$. 
For that, we need to show that the $D$-derivatives of 
$\bbG_s$ and $\bbH^n_s$ up to $\ell$-times are $\bbL_p$-bounded uniformly in 
$(u,v)\in\Lambda^0_n(\check{\sfd},\bar{d})$ and $n\in\bbN$, 
under the truncation by $\psi_n$. 
We see that this property holds for $\bbH^n_s$ by (B1)(ii). 
For $\bbG_s$, we verify the property as follows. 
The multiple $D$-derivative of $\bbG_s$ is a linear combination of terms of the form
\beas 
\bigg\{\prod_{{\colorr\alpha=1}}^{{\colorr \bar{\alpha}}} D_{{\colorr {\sf A}_\alpha}}
\bigg(\sum_{i_1,i_2=j}^k
\bbI^{i_1\wedge i_2}_su_{i_1}u_{i_2}\bigg)
\bigg\}\bbG_s
\quad {\colorr \bigg({\sf A}_\alpha = r_{a(\alpha-1)+1},...,r_{a(\alpha)}, 
\quad 1\leq a(1)\leq a(2)\leq\cdots\bigg) }
\eeas
that is bounded by 
\begin{en-text}
\beas 
\bigg\{\prod_\alpha 
\bigg(
\max_{i=j,...,k}\bigg|\frac{D_{{\colorr {\sf A}_\alpha}}\bbI^{i}_s}{T_{i}-s}\bigg|
\max_{i=j,...,k}\bigg[\frac{\bbI^{i}_s}{T_{i}-s}\bigg]^{-1}
\sum_{i_1,i_2=j}^k\bbI^{i_1\wedge i_2}_su_{i_1}u_{i_2}\bigg) \bigg\}\bbG_s.
\eeas
\end{en-text}
$\bbG_s$ times a polynomial ${\mathfrak p}$ of random variables 
\beas 
\max_{i=j,...,k}\bigg|\frac{D_{{\colorr {\sf A}_\alpha}}\bbI^{i}_s}{T_{i}-s}\bigg|,\quad
\max_{i=j,...,k}\bigg[\frac{\bbI^{i}_{s\vee T_{i-1}}}{T_{i}-(s\vee T_{i-1})}\bigg]^{-1},\quad
\max_{i=j,...,k}\bigg|\frac{\bbI^{i}_s}{T_{i}-s}\bigg|,\quad
|(u_i)_{i=j,...,k}|.
\eeas
Indeed, 
\beas &&
\bigg|D_{{\colorr {\sf A}_\alpha}}
\bigg(\sum_{i_1,i_2=j}^k
\bbI^{i_1\wedge i_2}_su_{i_1}u_{i_2}\bigg)\bigg|
\\&=&
\bigg|\frac{
\sum_{i_1,i_2=j}^k
D_{{\colorr {\sf A}_\alpha}}\bbI^{i_1\wedge i_2}_s/(T_{i_1\wedge i_2}-s)
u_{i_1}u_{i_2}
}
{
\sum_{i_1,i_2=j}^k
\bbI^{i_1\wedge i_2}_s/(T_{i_1\wedge i_2}-s)u_{i_1}u_{i_2}
}\bigg|
\bigg|\sum_{i_1,i_2=j}^k
\bbI^{i_1\wedge i_2}_s/(T_{i_1\wedge i_2}-s)u_{i_1}u_{i_2}\bigg|
\\&\leq&
\bigg|\bigg(D_{{\colorr {\sf A}_\alpha}}\bbI^{i_1\wedge i_2}_s/(T_{i_1\wedge i_2}-s)\bigg)_{i_1,i_2=j,...,k}\bigg|
\bigg|\bigg(\bbI^{i_1\wedge i_2}_s/(T_{i_1\wedge i_2}-s)\bigg)_{i_1,i_2=j,...,k}^{-1}\bigg|
\\&&\times \bigg|\sum_{i_1,i_2=j}^k
\bbI^{i_1\wedge i_2}_s/(T_{i_1\wedge i_2}-s)u_{i_1}u_{i_2}\bigg|,
\eeas
where we used 
\def\simleq{\ \raisebox{-.7ex}{$\stackrel{{\textstyle <}}{\sim}$}\ }
\beas 
|S^{-1/2}|&\simleq& \|S^{-1/2}\|_{op}
\>=\> \bigg(\sup_{v:|v|=1}S^{-1}[v^{\otimes2}]\bigg)^{1/2}
\>\leq\> |S^{-1}|^{1/2}
\eeas
for any non-degenerate symmetric matrix $S$. 
Moreover, the identity
\beas 
\det \bigg(\bbI^{i_1\wedge i_2}_s/(T_{i_1\wedge i_2}-s)\bigg)_{i_1,i_2=j,...,k}
&=&
\prod_{i=j,...,k}\frac{\bbI^{i}_{s\vee T_{i-1}}}{T_{i}-(s\vee T_{i-1})}
\eeas
can be used to estimate the inverse matrix in the above expression. 

\begin{en-text}
\beas 
\bigg\{\prod_\alpha 
\bigg(\sum_{i_1,i_2=j}^k 
\frac{D_{r_{g_{\alpha-1}},...,r_{g_\alpha}}\bbI^{i_1\wedge i_2}_s}{T_{i_1\wedge i_2}-s}
\bigg[\frac{\bbI^{i_1\wedge i_2}_s}{T_{i_1\wedge i_2}-s}\bigg]^{-1}
\bbI^{i_1\wedge i_2}_su_{i_1}u_{i_2}\bigg) \bigg\}\bbG_s,
\eeas
\end{en-text}
\noindent
The term ${\mathfrak p}\bbG_s$ is 
$\bbL_p$-bounded due to (B1)(i), (iii) and 
\beas 
\sup_{u\in\bbR^k,\omega,\atop s\in(T_{j-1},T_j)}
\bigg(\sum_{i_1,i_2=j}^k\bbI^{i_1\wedge i_2}_su_{i_1}u_{i_2}\bigg)^m\bbG_s
&<&\infty
\eeas
for every $m\in\bbN$ and $j=1,...,k$. 
\begin{en-text}
$L_p$-boundedness of 
$D$-derivatives of 
$\exp\big(-\half\big(C_1-C_s)[u^{\otimes2}]\big)$ by replacing the emerging  
$D^m(C_1-C_s)[u^{\otimes2}]$ by $(C_1-C_s)[u^{\otimes2}]$ times tame variables. 
\end{en-text}

If $j=k$, then this estimate is sufficient for our use. 
When $j<k$, 
we also use the non-degeneracy of 
the matrix 
\beas
\calm(T_{j+1},T_k) &=&
\bigg(
\half\int^{T_{i_1\wedge i_2}}_{T_j}K(t)^2dt
\bigg)_{i_1,i_2=j+1,...,k}
\eeas
and the estimate
\bea\label{20170310-1} 
(C_1-C_s)[u^{\otimes2}] 
&\geq&
\calm(T_{j+1},T_k)
\big[(u_{j+1},...,u_k)^{\otimes2}\big]
\eea
in order to obtain the decay $|(u_{j+1},...,u_k)|^{-(\check{\sfd}+1-\ep)}$. 
For (\ref{20170310-1}), we note that 
\beas 
\sum_{i_1,i_2=1}^k\int_s^{s\vee T_{i_1\wedge i_2}}K(t)^2dt\>u_{i_1}u_{i_2}
&=&
\int_s^1\bigg(\sum_{i=1}^k1_{[0,T_i]}(t)u_i\bigg)^2K(t)^2dt
\\&\geq&
\int_{T_j}^1\bigg(\sum_{i=1}^k1_{[0,T_i]}(t)u_i\bigg)^2K(t)^2dt
\\&=&
\sum_{i_1,i_2=j+1}^k\int_{T_j}^{T_{i_1\wedge i_2}}K(t)^2dt\>u_{i_1}u_{i_2}.
\eeas
By (\ref{290219-1}) we have
\bea\label{290219-2}
\det \calm(T_{j+1},T_k)^{-1} &\in& L_{\infty-},
\eea
and hence (\ref{20170310-1}) and (\ref{290219-2}) imply that 
\bea\label{290219-3} 
|(u_{j+1},...,u_k)|^m\exp\bigg(-\half\big(C_1-C_s)[u^{\otimes2}]\bigg)
&\leq&
{\sf C}_m\>\big|\calm(T_{j+1},T_k)^{-1}\big|^m
\eea
is $\bbL_{\infty-}$-bounded uniformly in $(u_{j+1},...,u_k)$ for every $m\in\bbN$. 
Finally, we may use one of the above estimates of the decay, depending on    
$|(u_1,...,u_j,v_1,...,v_{q})|\geq|(u_{j+1},...,u_k)|$ or not. 
\begin{en-text}
The estimate (\ref{290219-3}) also serves to show $L_p$-boundedness of $D$-derivatives of 
$\exp\big(-\half\big(C_1-C_s)[u^{\otimes2}]\big)$ by replacing the emerging  
$D^m(C_1-C_s)[u^{\otimes2}]$ by $(C_1-C_s)[u^{\otimes2}]$ times tame variables. 
\end{en-text}

Following the proof  of \cite[Theorem 4]{Yoshida2013}, i.e. 
the procedure (a)-(g) therein with the additional truncation 
\beas 
\psi_{n,s}^j &=& \psi\bigg(2\big[1+4\det \sigma_{(M^{n,1}_{T_1},...,M^{n,j-1}_{T_{j-1}},M^{n,j}_s,F\big)}
(s_n^j)^{-1}\big]^{-1}\bigg), 
\eeas
we obtain the desired decay of 
\beas 
n
\sum_{i=1}^{m_n}
\int_{t_{i-1}}^{t_i}1_{s\in(\tau^j_n,T_j)}\int_r^{t_i}  \mathbb E\bigg[\partial^{\sfa_0}K^n(t_{i-1})[\tti u]
\partial^{\sfa_1}D_r \bigg\{e^n_s(u)D_s\big(\Psi(u,v)\psi_n\big)\bigg\}\bigg]dsdr
\eeas
for $\alpha=0$. A similar estimate can be shown for a general $\alpha$. 

For $s\in(T_{j-1},\tau^j_n)$, we apply the integration-by-parts formula 
for 
\beas 
\big(M^{n,1}_{T_1},...,M^{n,j-1}_{T_{j-1}},F\big)
\eeas
to obtain the decay $|(u_1,...,u_{j-1},v_1,...,v_{q})|^{-(\check{\sfd}+1-\ep)}$. 
In order to obtain the decay $|(u_j,...,u_k)|^{-(\check{\sfd}+1-\ep)}$, we use 
the nondegeneracy of 
\beas
\calm(\tau^j_n,T_k) &=&
\bigg(
\half\int^{T_{i_1\wedge i_2}}_{\tau^j_n}K(t)^2dt
\bigg)_{i_1,i_2=j,\ldots,k}.
\eeas
Then we repeat a similar procedure as in the previous case to obtain the desired decay. 
We deduce (\ref{170208-2}) by combining the above estimates.

\subsection{Proof of Theorem \ref{th2n}} 
We state the decompositions in the differential form for the ease of exposition. 
Applying Taylor expansion we conclude that
\begin{align} \label{stochexpansion}
dV_t^n &=  \sqrt{n}\left( a(X^n_{\varphi_n(t)}) - a(X_t)  \right) dt + 
\sqrt{n}\left( b(X^n_{\varphi_n(t)})  - b(X_t)   \right) dW_t \\[1.5 ex]
&= \sqrt{n} \left( a(X^n_{t}) - a(X_t)  \right) dt +  \sqrt{n}\left(  a(X^n_{\varphi_n(t)}) - a(X_t^n)  \right) dt 
\nonumber \\[1.5 ex]
&+\sqrt{n} \left( b(X^n_{t})  - b(X_t) \right) dW_t +  \sqrt{n}\left(  b(X^n_{\varphi_n(t)}) -b(X_t^n)  \right) dW_t  \nonumber \\[1.5 ex]
&=\left( \left(a'(X_t) + \tilde a_t^n \right) V_t^n + \tilde a_t^{\prime n} \right)dt \nonumber \\[1.5 ex]
&+ \left(\left(b'(X_t) +  \tilde b_t^n\right) V_t^n - \sqrt{n} bb'(X^n_{\varphi_n(t)})
(W_t - W_{\varphi_n(t)}) +  \tilde b_t^{\prime n} \right) dW_t,
\nonumber
\end{align} 
where the processes $\tilde a^n, \tilde a^{\prime n}, \tilde b^n, \tilde b^{\prime n}$ are defined as
\begin{align*}
&\tilde a^n_t = a'(\widetilde{X}_t^n) -  a'(X_t), \qquad \tilde a^{\prime n}_t= \sqrt{n}\left(  a(X^n_{\varphi_n(t)}) - a(X_t^n)  \right), \\[1.5 ex]
& \tilde b^n_t = b'(\widetilde{X}_t^{\prime n}) -  b'(X_t), \qquad \tilde b^{\prime n}_t=
\sqrt{n} \left( b(X^n_{\varphi_n(t)}) -b(X_t^n) +  bb'(X^n_{\varphi_n(t)})
(W_t - W_{\varphi_n(t)}) \right),
\end{align*}
and $\widetilde{X}_t^n, \widetilde{X}_t^{\prime n}$ are certain random variables with 
$|\widetilde{X}_t^n - X_t| \leq |X_t^n - X_t|$, $|\widetilde{X}_t^{\prime n} - X_t| \leq |X_t^n - X_t|$.
In particular, it holds that $\widetilde{X}^n \ucp X$ and $\widetilde{X}^{\prime n} \ucp X$.
Using Lemma \ref{lem1} we thus can write 
\begin{align*}
V_t^n = \Sigma_t^n \left(\int_0^t   (\Sigma_s^n)^{-1}  \left( \tilde{a}_s^{\prime n} 
- \left(b'(X_t) +  \tilde b_t^n\right)  \left (\tilde{b}_s^{\prime n} - \sqrt{n} bb'(X_{\varphi_n(s)}) (W_s - W_{\varphi_n(s)}) \right) \right)ds 
\right. \\
\left. + 
 \int_0^t   (\Sigma_s^n)^{-1} \left (\tilde{b}_s^{\prime n} - \sqrt{n} bb'(X_{\varphi_n(s)}) (W_s - W_{\varphi_n(s)}) \right) dW_s \right),
\end{align*}
where the process $\Sigma^n$ is defined by
\begin{align} \label{sigmatn}
\Sigma_t^n &= 
\exp \left( \int_0^t (b'(X_s) + \tilde{b}_s^n) dW_s  \right. \\
& \left.+   \int_0^t \left(a'(X_s) + \tilde{a}_s^n - 
\frac12 \left( b'(X_s)+ \tilde{b}_s^n  \right)^2  \right) ds \right) \nonumber 
\end{align}
Comparing the representation of $V_t^n$ with \eqref{barvn}, we just need to show that 
\begin{align} 
& \label{ucp1} \Sigma^n \ucp \Sigma, \\
&\label{ucp2} \int_0^t \Sigma_s^{-1}  \tilde{a}_s^{\prime n}  ds \ucp 0, \\
& \label{ucp3}  \int_0^t   \Sigma_s^{-1} \tilde{b}_s^{\prime n} dW_s \ucp 0, \\
& \label{ucp4}  \int_0^t   \Sigma_s^{-1} \left(b'(X_t) +  \tilde b_t^n\right)  \left (\tilde{b}_s^{\prime n} - \sqrt{n} bb'(X_{\varphi_n(s)}) (W_s - W_{\varphi_n(s)}) \right) ds \ucp 0,
\end{align} 
where the process $\Sigma$ has been defined in \eqref{vexplicit}. Since both $\tilde{a}_s^n$ 
and $\tilde{b}_s^n$ are bounded as assumed in the beginning of Section  \ref{sec7},  and $\tilde{a}^n \ucp 0$, $\tilde{b}^n \ucp 0$ (because 
$\widetilde{X}^n \ucp X$, $\widetilde{X}^{\prime n} \ucp X$ and $a,b \in C^2(\R)$), we readily 
deduce the convergence at \eqref{ucp1}. To show the convergence at  \eqref{ucp2} we use the decomposition $\tilde{a}_s^{\prime n}= \tilde{a}_s^{\prime n,1} + \tilde{a}_s^{\prime n,2}$ with 
\begin{align*}
\tilde{a}_s^{\prime n,1}&= - \sqrt{n} a'(X^n_{\varphi_n(s)}) \int_{\varphi_n(s)}^s b(X^n_{\varphi_n(u)}) 
dW_u, \\
\tilde{a}_s^{\prime n,2} &= \sqrt{n} \left( a'(X^{\prime \prime n}_{s})  - a'(X^n_{\varphi_n(s)}) \right) \left( X^n_{\varphi_n(s)} - X_s^n  \right) \\
& - \sqrt{n} a'(X^n_{\varphi_n(s)}) \int_{\varphi_n(s)}^s a(X^n_{\varphi_n(u)}) 
du,
\end{align*}
where $X^{\prime \prime n}_{s}$ is a certain random variable with 
$|X^{\prime \prime n}_{s} - X_s^n| \leq |X_s^n - X^n_{\varphi_n(s)}|$. Since $X^{\prime \prime n}_{s}
\ucp X_s$ and all involved objects are assumed to be bounded, we conclude by \eqref{BDG} that
\[
\E[|\tilde{a}_s^{\prime n,2}|] \leq C \epsilon_n
\] 
with $\epsilon_n \to 0$ as $n \to \infty$. Thus, we obtain 
\begin{align*}
\int_0^t \Sigma_s^{-1}  \tilde{a}_s^{\prime n,2}  ds \ucp 0
\end{align*}
by an application of \eqref{negligibility}. Now, we notice that $\E[\tilde{a}_s^{\prime n,1}| 
~\mathcal F_{(i-1)/n}] = 0$ and $\E[|\tilde{a}_s^{\prime n,1}|^2 ] \leq C$. Thus, we deduce that 
\begin{align*} 
\int_0^t \Sigma_s^{-1}  \tilde{a}_s^{\prime n,1}  ds =
\int_0^t \Sigma_{\varphi_n(s)}^{-1}  \tilde{a}_s^{\prime n,1}  ds + \int_0^t 
\left(\Sigma_s^{-1} - \Sigma_{\varphi_n(s)}^{-1}   \right) \tilde{a}_s^{\prime n,1}  ds
 \ucp 0,
\end{align*} 
which follows by a combination of \eqref{negligibility} and \eqref{doob}. Indeed, it holds that 
\[
\int_0^t \Sigma_{\varphi_n(s)}^{-1}  \tilde{a}_s^{\prime n,1}  ds = -\sqrt n \sum_{i=1}^{[nt]}
\Sigma_{\frac{i-1}{n}}^{-1} a' (X^n_{\frac{i-1}{n}}) b' (X^n_{\frac{i-1}{n}}) 
\int_{\frac{i-1}{n}}^{\frac in} (W_s -W_{\frac{i-1}{n}} )ds + o_{\mathbb P}(1),
\]
and \eqref{doob} can be applied to the last line. 
Consequently, we have \eqref{ucp2}. Finally, we show the convergence at \eqref{ucp3}.
Observe the decomposition  
\begin{align*}
\tilde{b}_s^{\prime n} &= \sqrt{n} \left( b'(X^{\prime \prime n}_{s})  - b'(X^n_{\varphi_n(s)}) \right) \left( X^n_{\varphi_n(s)} - X_s^n  \right) \\
& - \sqrt{n} b'(X^n_{\varphi_n(s)}) \left(\int_{\varphi_n(s)}^s a(X_u^n) du+ \int_{\varphi_n(s)}^s b(X_u^n) - b(X_{\varphi_n(s)}^n) dW_u \right).
\end{align*}
As for the term $\tilde{a}_s^{\prime n,2}$ we deduce that $\E[|\tilde{b}_s^{\prime n}|] 
\leq C \epsilon_n$ with $\epsilon_n \to 0$ as $n \to \infty$. Hence, we obtain \eqref{ucp3}.
The proof of \eqref{ucp4} combines the proof methods of  \eqref{ucp2} and \eqref{ucp3}.
Consequently, 
\[
\sup_{t \in [0,1]} |V_t^n -  \overline{V}_t^n| \toop 0,
\]
which completes the proof of Theorem  \ref{th2n}. \qed

\subsection{Proof of Theorem \ref{th3}} 
The derivation of the second order stochastic expansion is more involved than the expansion of Theorem \ref{th2n}, but the underlying methodology is similar. For simplicity of exposition we sometimes use the same notations as in the previous section although they might have a different meaning. 
Instead of the first order approximation in the last line of 
\eqref{stochexpansion}, we may further develop
\begin{align}  \label{2expansion}
dV_t^n &=  
a'(X_t) V_t^ndt + (b'(X_t) V_t^n - \sqrt{n} bb'(X^n_{\varphi_n(t)})(W_t - W_{\varphi_n(t)})) dW_t \\[1.5 ex]
&+ \left(\frac{1}{2 \sqrt{n}} a''(X_t) (V_t^n)^2 - \sqrt{n} 
 ba'(X^n_{\varphi_n(t)})
(W_t - W_{\varphi_n(t)})  \right.  \nonumber \\[1.5 ex]
 &\left. -\sqrt{n}aa'(X^n_{\varphi_n(t)}) 
(t - \varphi_n(t)) - \frac{\sqrt{n}}{2} b^2a''(X^n_{\varphi_n(t)}) (W_t - W_{\varphi_n(t)})^2
\right) dt\nonumber \\[1.5 ex]
&+ \Big(
\frac{1}{2 \sqrt{n}} b''(X_t) (V_t^n)^2 -\frac{\sqrt{n}}{2} 
 b^2b'' (X_{\varphi_n(t)}^n)
(W_t - W_{\varphi_n(t)})^2 \nonumber  \\[1.5 ex] 
&  
	- \sqrt{n}ab' (X_{\varphi_n(t)}^n)(t - \varphi_n(t)) \Big) dW_t 
+ \tilde{a}_t^n dt + \tilde{b}_t^n dW_t \nonumber,
\end{align} 
where  $ \tilde{a}_t^n$ and $ \tilde{b}_t^n$ are stochastic processes, whose negligibility in the involved asymptotic expansions is shown in exactly the same manner as in   \eqref{ucp1}-\eqref{ucp4} (although these terms have a different meaning in this subsection).

Now, recall the definition of the first order approximation $ \overline{V}_t^n$ at \eqref{barvn}. By Lemma \ref{lem1} this process satisfies the
stochastic differential 
equation 
\begin{align*} 
d\overline{V}_t^n &= 
a'(X_{t}) \overline{V}_t^n dt + b'(X_{t}) \overline{V}_t^n dW_t
 - \sqrt{n} \Sigma_t \Sigma_{\varphi_n(t)}^{-1} bb'(X_{\varphi_n(t)}^n)
(W_t - W_{\varphi_n(t)}) dW_t \\[1.5 ex]
& - \sqrt{n}\Sigma_t \Sigma_{\varphi_n(t)}^{-1}b'(X_t) bb'(X^n_{\varphi_n(t)})(W_t - W_{\varphi_n(t)})) dt.
\end{align*} 
Observing the definition of the stochastic process $dR_t^n= R_t^n(1) dt + R_t^n(2) dW_t$ 
at  \eqref{rn}, we deduce by Lemma \ref{lem1} and the negligibility of the terms $ \tilde{a}_t^n$, $ \tilde{b}_t^n$ the decomposition
\[
V_t^n -  \overline{V}_t^n = \Sigma_t \int_0^t \Sigma_s^{-1} 
\left( \left(R_s^n(1) - b'(X_s) R_s^n(2))\right) ds + 
 R_s^n(2) dW_s\right)  + o_{\mathbb P}(n^{-1/2}), 
\]  
where $\Sigma$ has been defined in \eqref{vexplicit}. This finishes the proof of Theorem \ref{th3}. \qed

\subsection{Proof of Theorem \ref{maintheorem}}\label{180522-9}
We will verify conditions 
$(B1)$, $(B2)$ and $(B3)$ for Theorem \ref{20170206-1} 
under $(A)$, $(C1)$, $(C2)$ and $(C3)$. 
Recall that 
\beas 
K(s)=-\Sigma_s^{-1}bb'(X_s), 
\eeas 
$\ell=2k+\overline{q}+8$, $\ell_*=2[\overline{q}/2]+4$ and we are assuming that $a$, $b$ are in $C^\infty(\bbR)$ and 
all their derivatives of positive order are bounded.
As mentioned just before assumption $(A)$, 
the functionals $c_\alpha(z)$ in the representation (\ref{180520-2}) 
of the full random symbol $\sigma$ 
and also in (\ref{20170404-1}) 
are associated with 
$\underline{\sigma}$ of (\ref{Lsigma}) and 
$\overline{\sigma}$ of (\ref{170207-5}). 

Conditions $(B1)$(i), (ii) are obvious. 
Condition $(B1)$(iii) is assumed by $(C1)$. 
In the present situation, $\widehat{F}_n=0$ since $F_n=F$. 
Condition $(B2)$(i) follows from $(A)$ and $(C2)$(i). 
$(B2)$(ii) will be checked later after constructing $s_n$. 
Condition $(B2)$(iii) is already obtained in (\ref{180522-1}). 
The property $(B2)$(iv) has been observed to derive 
the expression (\ref{Lsigma}). 

We shall consider 
non-degeneracy of the Malliavin covariance matrix 
$\sigma_{(\bbX_1,\bbX_2)}$ of $(\bbX_1,\bbX_2)$, where 
\beas 
\bbX_1 &=& 
\big(M^{n,1}_{S_1},...,M^{n,j-1}_{S_{j-1}},M^{n,j}_{S_j})
 \quad\text{and}\quad
\bbX_2 \>=\> \big(\Sigma_{T_1},...,\Sigma_{T_k}\big)
\eeas
for $S_1=T_1,...,S_{j-1}=T_{j-1}$ and $S_j$ is either $s\in[(T_{j-1}+T_j)/2,T_j]$. 
We will estimate the Malliavin covariance matrix $\sigma_{(\bbX_1,\bbX_2)}$. 
%
\begin{en-text}
\beas 
\bbX &=& \bigg(
\big(M_{T_1}-M_{T_0},\Sigma_{T_1}-\Sigma_{T_0}\big),...,
\big(M_{T_{j-1}}-M_{T_{j-2}},\Sigma_{T_{j-1}}-\Sigma_{T_{j-2}}\big), 
\\&&
\big(M_s-M_{T_{j-1}},\Sigma_{T_j}-\Sigma_{T_{j-1}}\big), 
\Sigma_{T_{j+1}}-\Sigma_{T_j},...,\Sigma_{T_k}-\Sigma_{T_{k-1}}\bigg)
\eeas
for $s\in((T_{j-1}+T_j)/2,T_j)$. 
Remark that $\big(M^n_{T_1}-M_{T_0},\Sigma_{T_1}-\Sigma_{T_0}\big)=(M_{T_1},\Sigma_{T_1})$. 
\beas 
\bbX &=& \bigg(\big(M^n_{T_1},\Sigma_{T_1}\big),...,
\big(M^n_{T_{j-1}},\Sigma_{T_{j-1}}\big), 
\big(M^n_s,\Sigma_{T_j}\big), 
\Sigma_{T_{j+1}}...,\Sigma_{T_k}\bigg)
\eeas
for $s\in((T_{j-1}+T_j)/2,T_j)$. 
\end{en-text}
Let $\theta_i=i/n$. 
Let 
\beas 
\eta_i(t) &=& \sqrt{n}\big(W(\theta_i\wedge t)-W(\theta_{i-1}\wedge t)\big)
\eeas
and 
\beas 
\xi_i(t) &=& n\bigg(\big(W(\theta_i\wedge t)-W(\theta_{i-1}\wedge t)\big)^2-\big(\theta_i\wedge t-\theta_{i-1}\wedge t\big)\bigg).
\eeas
Then, as in \cite{Yoshida2012a}, we have 
\beas
D_rM^{n,\mu}_{S_\mu}
&=&
\sum_{i=1}^n2K(\theta_{i-1})\eta_i(S_\mu)
1_{(\theta_{i-1}\wedge S_\mu,\theta_i\wedge S_\mu]}(r)
\\&&
+n^{-1/2}\sum_{i=1}^{n-1}\bigg(\sum_{i'=i+1}^nD_rK(\theta_{i-1})
\xi_{i'}(S_\mu)\bigg)1_{(\theta_{i-1}\wedge S_\mu,\theta_i\wedge S_\mu]}(r)
\eeas
for $\mu=1,...,j$. 
Therefore, 
\beas 
\sigma(n,\mu_1,\mu_2)&:=& \langle DM^n_{S_{\mu_1}},DM^n_{S_{\mu_2}} \rangle_{\mathbb H}
\\&=&
\sum_{i=1}^n\int_{\theta_{i-1}}^{\theta_i}
\bigg[2K(\theta_{i-1})\eta_i(S_{\mu_1})+n^{-1/2}\sum_{i'=i+1}^nD_rK(\theta_{i-1})\xi_{i'}(S_{\mu_1})\bigg]
1_{[0,S_{\mu_1}]}(r)
\\&&
\qquad\quad\times
\bigg[2K(\theta_{i-1})\eta_i(S_{\mu_2})+n^{-1/2}\sum_{i'=i+1}^nD_rK(\theta_{i-1})\xi_{i'}(S_{\mu_2})\bigg]
1_{[0,S_{\mu_2}]}(r)
dr
\\&&
+O_{\mathbb L^p}(n^{-1/2})
\\&=&
\tilde{\sigma}(n,\mu_1,\mu_2)+O_{\mathbb L^p}(n^{-1/2})
\eeas
for $\mu_1,\mu_2=1,...,j$, 
where $\tilde{\sigma}(n,\mu_1,\mu_2)=\tilde{\sigma}_1(n,\mu_1,\mu_2)+\tilde{\sigma}_2(n,\mu_1,\mu_2)$ with 
\beas
\tilde{\sigma}_1(n,\mu_1,\mu_2) 
&=& 
\sum_{i=1}^n\int_{\theta_{i-1}}^{\theta_i}
\big(2K(\theta_{i-1})\eta_i(S_{\mu_1})\big)1_{[0,S_{\mu_1}]}(r)
\big(2K(\theta_{i-1})\eta_i(S_{\mu_2})\big)1_{[0,S_{\mu_2}]}(r)dr
\eeas
and
\beas
\tilde{\sigma}_2(n,\mu_1,\mu_2)
&=&
\sum_{i=1}^n\int_{\theta_{i-1}}^{\theta_i}
\bigg(n^{-1/2}\sum_{i'=i+1}^nD_rK(\theta_{i-1})\xi_{i'}(S_{\mu_1})\bigg)1_{[0,S_{\mu_1}]}(r)
\\&&
\qquad\quad\times
\bigg(n^{-1/2}\sum_{i'=i+1}^nD_rK(\theta_{i-1})\xi_{i'}(S_{\mu_2})\bigg)1_{[0,S_{\mu_2}]}(r)dr
\eeas
for $\mu_1,\mu_2=1,...,j$. 
Moreover,  for $G=(G^\nu)_{\nu=1,...,\overline{q}}$,
\beas 
\sigma(n,\mu,\nu)
&:=&
\langle DM^n_{S_\mu},D G^\nu\rangle_{\mathbb H}
\\&=&
\sum_{i=1}^n\int_{\theta_{i-1}}^{\theta_i}
\bigg[2K(\theta_{i-1})\eta_i(S_{\mu})+n^{-1/2}\sum_{i'=i+1}^nD_rK(\theta_{i-1})\xi_{i'}(S_{\mu})\bigg]
1_{[0,S_{\mu_1}]}(r)
\\&&
\qquad\quad\times D_r G^{\nu} dr+O_{\mathbb L^p}(n^{-1/2})
\\&=&
\tilde{\sigma}(n,\mu,\nu)+O_{\mathbb L^p}(n^{-1/2}),
\eeas
where
\beas 
\tilde{\sigma}(n,\mu,\nu) &=& 
\sum_{i=1}^n\int_{\theta_{i-1}}^{\theta_i}
\bigg(n^{-1/2}\sum_{i'=i+1}^nD_rK(\theta_{i-1})\xi_{i'}(S_{\mu})\bigg)
1_{[0,S_{\mu}]}(r)
D_r G^\nu dr
\eeas
Let 
\beas 
\tilde{\sigma}(n,\nu_1,\nu_2)
&=& 
\int _0^1
D_r G^{\nu_1}D_r G^{\nu_2}dr.
\eeas
Then it is easy to see that the matrix 
\beas
\left[
\begin{array}{cc}
\big(\tilde{\sigma}_2(n,\mu_1,\mu_2)\big)&\big(\tilde{\sigma}(n,\mu,\nu)\big)
\\ 
\big(\tilde{\sigma}(n,\mu,\nu)\big)^\star&\big(\tilde{\sigma}(n,\nu_1,\nu_2)\big)
\end{array}
\right]
\eeas
is nonnegative definite. 
As we will see, the matrix $\big(\tilde{\sigma}(n,\nu_1,\nu_2)\big)$ is 
positive definite almost surely. 
Therefore, 
\beas 
\begin{array}{c}
\big(\tilde{\sigma}_2(n,\mu_1,\mu_2)\big)-
\big(\tilde{\sigma}(n,\mu,\nu)\big)\big(\tilde{\sigma}(n,\nu_1,\nu_2)\big)^{-1}\big(\tilde{\sigma}(n,\mu,\nu)\big)^\star\end{array}
\eeas
is nonnegative definite, and hence 
\beas&&
\det\left[
\begin{array}{cc}
\big(\tilde{\sigma}(n,\mu_1,\mu_2)\big)&\big(\tilde{\sigma}(n,\mu,\nu)\big)
\\ 
\big(\tilde{\sigma}(n,\mu,\nu)\big)^\star&\big(\tilde{\sigma}(n,\nu_1,\nu_2)\big)
\end{array}
\right]
\\&=&
\det\left[
\begin{array}{c}
\big(\tilde{\sigma}(n,\mu_1,\mu_2)\big)-
\big(\tilde{\sigma}(n,\mu,\nu)\big)\big(\tilde{\sigma}(n,\nu_1,\nu_2)\big)^{-1}\big(\tilde{\sigma}(n,\mu,\nu)\big)^\star\end{array}
\right]
\\&&
\times\det \big(\tilde{\sigma}(n,\nu_1,\nu_2)\big)
\\&\geq&
\det\big(\tilde{\sigma}_1(n,\mu_1,\mu_2)\big)
\det \big(\tilde{\sigma}(n,\nu_1,\nu_2)\big)
\>=:\>\calm_n.
\eeas

\begin{en-text}
We consider the quadratic form $Q$ given by 
\beas 
\sum_{\mu_1,\mu_2=1}^j u_{\mu_1}u_{\mu_2} \tilde{\sigma}(n,\mu_1,\mu_2)
+2\sum_{\mu}^j\sum_{\nu=1}^k u_\mu v_\nu \tilde{\sigma}(n,\mu,\nu)
+\sum_{\nu_1,\nu_2=1}^k v_{\nu_1}v_{\nu_2}\tilde{\sigma}(n,\nu_1,\nu_2)
\eeas
for $u_1,...,u_j,v_1,...,v_k\in\bbR$.  
Then it is possible to show $Q\geq0$ for any $u_1,...,u_j,v_1,...,v_k\in\bbR$ 
by the Schwarz inequality and the variance inequality. 
Therefore the kernel matrix of $Q$ is non-negative, and hence $\det Q\geq0$. 
By decomposition of the determinant 
\beas 
\det\left[
\begin{array}{cc}
A_1+A_2 & B\\
C & D
\end{array}
\right]
=
\det\left[
\begin{array}{cc}
A_1& B\\
O & D
\end{array}
\right]
+
\det
\left[
\begin{array}{cc}
A_2 & B\\
C & D
\end{array}
\right],
\eeas
we see 
\beas
\det\left[
\begin{array}{cc}
\big(\tilde{\sigma}(n,\mu_1,\mu_2)\big)&\big(\tilde{\sigma}(n,\mu,\nu)\big)
\\ 
\big(\tilde{\sigma}(n,\mu,\nu)\big)^\star&\big(\tilde{\sigma}(n,\nu_1,\nu_2)\big)
\end{array}
\right]
\eeas
is bounded from below by 
\beas
\calm_n &:=& \det\bigg[ 
\frac{1}{n}\sum_{i=1}^n\bigg[2a(X(\theta_{i-1})\eta_i(\varphi_n(S_{\mu_1}\wedge S_{\mu_2}));\mu_1,\mu_2\bigg]^2
\bigg]
\times 
\det\bigg[\tilde{\sigma}(n,\nu_1,\nu_2)\big);\nu_1,\nu_2\bigg].
\eeas
\end{en-text}
Now $\calm_n$ converges in $\bbL_{\infty-}$ to 
\beas
\calm_\infty &:=& \det\bigg[ 
\int_0^{S_{\mu_1}\wedge S_{\mu_2}}
4K(t)^2dt\bigg]_{\mu_1,\mu_2=1,...,j}
\times 
\det\bigg[\int _0^1
D_r G^{\nu_1}D_r G^{\nu_2}dr\bigg]_{\nu_1,\nu_2=1,...,\overline{q}}
\eeas
with rate $n^{-1/2}$. 
Define $s_n^j$ by 
\beas 
s_n^j := \half \calm_\infty.
\eeas
\begin{en-text}
We see 
\beas
\int _0^1D_r\Sigma_{T_{\nu_1}}D_r\Sigma_{T_{\nu_2}}dr
&=&
\int _0^{T_{\nu_1}\wedge T_{\nu_2}}D_r\Sigma_1D_r\Sigma_1dr\koko
\eeas
\end{en-text}
Then 
$
\sup_{n\in\bbN}\big\|s_n^j\big\|_{\ell,p}
< \infty
$ 
for every $p>1$ and every $j$, so that $(B2)$(ii) holds additionally by $(A)$. 
But $\calm_\infty$ is non-degenerate, i.e.
\bea\label{180522-8}
\calm_\infty^{-1}&\in& \bbL_{\infty-}
\eea 
due to $(C1)$ and $(C3)$. 
This shows $(B3)$(ii). 
Moreover, the estimate $\calm_n-\calm_\infty=O_{\bbL_{\infty-}}(n^{-1/2})$ 
and (\ref{180522-8}) proves $(B3)$(i). Hence,
the proof of Theorem \ref{maintheorem} is completed. 
\qed


\subsection*{Acknowledgment}
Mark Podolskij and Bezirgen Veliyev acknowledge financial support from the project 
``Ambit fields: probabilistic properties and statistical inference'' funded by Villum Fonden (No. 11745)
and from
CREATES funded by the Danish
National Research Foundation. 
The research of Nakahiro Yoshida is supported in part by 
Japan Society for the Promotion of Science Grants-in-Aid for Scientific Research 
No. 17H01702 (Scientific Research), 
No. 26540011 (Challenging Exploratory Research); 
Japan Science and Technology Agency CREST JPMJCR14D7; 
and by a Cooperative Research Program of the Institute of Statistical Mathematics.

\end{document}